\documentclass[preprint,1p]{elsarticle}

\makeatletter
 \def\ps@pprintTitle{%
 	\let\@oddhead\@empty
 	\let\@evenhead\@empty
 	\def\@oddfoot{\footnotesize\itshape
 		{} \hfill\today}%
 	\let\@evenfoot\@oddfoot
 }
\makeatother

\usepackage{lipsum} 
\ExplSyntaxOn
\NewDocumentEnvironment{multiequation}{b}
 {
  \vantiempham:n { #1 }
 }
 {}

\seq_new:N \l__vantiempham_md_seq

\cs_new_protected:Nn \vantiempham:n
 {
  \seq_set_split:Nnn \l__vantiempham_md_seq { \\ } { #1 }
  $$
  \seq_map_function:NN \l__vantiempham_md_seq \__vantiempham_md_item:n
  $$
 }
\cs_new_protected:Nn \__vantiempham_md_item:n
 {
  \begin{minipage}{\dim_eval:n {\displaywidth/(\seq_count:N \l__vantiempham_md_seq)} }
  \begin{equation}
  \cs_set_eq:Nc \label { ltx@label }
  #1
  \vphantom{\seq_use:Nn \l__vantiempham_md_seq {}}
  \end{equation}
  \end{minipage}
 }

\ExplSyntaxOff  
\usepackage[unicode]{hyperref}
\usepackage{enumitem}
\makeatletter
\providecommand*{\cupdot}{%
  \mathbin{%
    \mathpalette\@cupdot{}%
  }%
}
\newcommand*{\@cupdot}[2]{%
  \ooalign{%
    $\m@th#1\cup$\cr
    \hidewidth$\m@th#1\cdot$\hidewidth
  }%
}
\makeatother

\usepackage[makeroom]{cancel}
 
\usepackage{soul}

\usepackage{latexsym}
\usepackage{indentfirst}
\usepackage{amsxtra}
\usepackage{amssymb}
\usepackage{amsthm}
\usepackage{amsmath}
\usepackage{mathrsfs} 

\usepackage{xcolor}
\usepackage{color}

\usepackage{amsfonts}
\usepackage[capitalise]{cleveref}

\numberwithin{equation}{section}

\newtheorem{theor}{Theorem}[section]
\newtheorem{prop}[theor]{Proposition}
\newtheorem{lemma}[theor]{Lemma}
\newtheorem{convention}[theor]{Convention}
\newtheorem{cor}[theor]{Corollary}
\theoremstyle{definition}               
\newtheorem{defin}[theor]{Definition}
\newtheorem{ex}[theor]{Example}
\newtheorem{exs}[theor]{Examples}
\newtheorem{rem}[theor]{Remark}

\newcommand{\vd}{\vdash}
\newcommand{\dv}{\dashv}
\newcommand{\overeq}[1]{\overset{#1}{=}}
\newcommand{\tr}{\triangleright}

\DeclareMathOperator{\Sym}{Sym}
\DeclareMathOperator{\Aut}{Aut}
\DeclareMathOperator{\Conj}{Conj}

\DeclareMathOperator{\End}{End}
\DeclareMathOperator{\Map}{Map}
\DeclareMathOperator{\id}{id}

\DeclareMathOperator{\lcm}{lcm}

\usepackage{lastpage}

\makeatletter
\def \@evenhead {\thepage\ of \pageref{LastPage} \hfil \slshape \leftmark } 
\def \@oddhead {{\slshape \rightmark }\hfil \thepage\ of \pageref{LastPage}} 
\makeatother

\begin{document}
\begin{frontmatter}

\title{Generalized digroups, di-skew braces, and solutions of the set-theoretic Yang--Baxter equation}
	\tnotetext[mytitlenote]{This work was partially supported by the Dipartimento di Matematica e Fisica ``Ennio De Giorgi'' - Università del Salento. The authors are members of GNSAGA (INdAM).	
}

\author{Andrea ALBANO}
\ead{andrea.albano@unisalento.it}

\author{Paola STEFANELLI}
\ead{paola.stefanelli@unisalento.it}

\address{Dipartimento di Matematica e Fisica ``Ennio De Giorgi", Universit\`a del Salento, \\ Via Provinciale Lecce-Arnesano, 73100 Lecce (Italy)}

\begin{abstract}
We introduce a novel algebraic structure called \emph{di-skew brace} by which we show that generalized digroups systematically yield bijective, non-degenerate solutions to the set-theoretic Yang--Baxter equation. 
We study the structural properties of these solutions with a particular focus on their left derived shelves, which belong to the class of \emph{conjugation racks}.
Consistently, we show that these solutions belong to a broader class that includes skew brace solutions.
In particular, we prove that each such solution can be decomposed as a \emph{hemi-semidirect} product of a skew brace solution endowed with a certain compatible action on the idempotents of the associated di-skew brace structure.
Finally, we provide concrete instances of these solutions through a suitable notion of \emph{averaging operators} on groups.
\end{abstract}

\end{frontmatter}

\section*{Introduction}

The problem of studying solutions of the set-theoretic Yang--Baxter equation first arose in a paper by Drinfel'd \cite{Dr92} and since then has attracted much interest with the aim of determining their algebraic properties and towards the task of classifying them.
Specifically, if $D$ is a set, then a map $r:D\times D\to D\times D$ is said to be  a \emph{solution of the set-theoretic Yang--Baxter equation on $D$} (briefly a \emph{solution}) if the identity
\begin{align}\label{ybe_eq}\tag{YBE}
\left(r\times\id_D\right)\left(\id_D\times r\right)\left(r\times\id_D\right)
=
\left(\id_D\times r\right)\left(r\times\id_D\right)\left(\id_D\times r\right)
\end{align}
is satisfied.
One of the most common way to encode the braiding of a solution $r$ on $D$ is to denote it by $\left(D,r\right)$ and write $r\left(a, b\right) = \left(\lambda_{a}\left(b\right), \rho_{b}\left(a\right)\right)$, where $\lambda_{a}$ and $\rho_{b}$ are maps from $D$ into itself, for all $a,b\in D$. 
To avoid overloading the notation, these maps are used without further specifications on the particular sets where they are defined, which will be clear from the context.
With this notation, a straightforward computation ensures that \eqref{ybe_eq} holds if and only if the following identities
\begin{align}
     &\label{first} \lambda_a\lambda_b(c)=\lambda_{\lambda_a\left(b\right)}\lambda_{\rho_b\left(a\right)}\left(c\right)\tag{Y1}\\
    &  \label{second}\lambda_{\rho_{\lambda_b\left(c\right)}\left(a\right)}\rho_c\left(b\right)=\rho_{\lambda_{\rho_b\left(a\right)}\left(c\right)}\lambda_a\left(b\right)\tag{Y2}\\
    &\label{third}\rho_c\rho_b(a)=\rho_{\rho_c\left(b\right)}\rho_{\lambda_b\left(c\right)}\left(a\right)\tag{Y3}\end{align} 
hold, for all $a,b,c \in D$. Moreover, a solution $r$ is \textit{left non-degenerate} if $\lambda_a \in \Sym_D$, for all $a \in D$, \textit{right non-degenerate} if $\rho_a \in \Sym_D$, for all $a \in D$, and \emph{non-degenerate} if it is both left and right non-degenerate. 
In addition, a solution $r$ is \textit{bijective} if it is a permutation of $D\times D$. 

Bijective non-degenerate solutions are among the first on which algebraic approaches have been applied to obtain significant results, for example, as in \cite{ESS99}, \cite{GaMa08}, \cite{GaVa98}, \cite{LuYZ00}, and \cite{So00}. 
{Even though in the latter and many other related works bijectivity is considered an assumption, it is now known that a non-degenerate solution is automatically bijective, thanks to \cite{JePi25}.}
More recent papers shifted the attention to broader classes of solutions in an effort to drop one non-degeneracy assumption. In particular, \cite{Le18} shows how self-distributive structures, commonly named \emph{shelves}, turn out to be relevant for a deeper understanding of left non-degenerate solutions. If $D$ is a non-empty set and $\triangleright$ a binary operation on $D$, then the pair $\left(D,\triangleright\right)$ is a \emph{shelf} if the identity
$a\triangleright\left(b\triangleright c\right) = \left(a\triangleright b\right)\triangleright\left(a\triangleright c\right)$ is satisfied, for all $a,b,c\in D$.
In this framework, the notion of Drinfel'd twist connected to a shelf structure has been investigated in \cite{DoRySte24} and has allowed to provide a description of all left non-degenerate solutions.

The main aim of this work is to provide bijective non-degenerate solutions in terms of Drinfel'd twists on \emph{conjugation racks}, which are self-distributive structures naturally associated with \emph{digroups}.
More generally, we recall that a \emph{disemigroup} ${(D,\vd,\dv)}$ is the datum of a set $D$ equipped with two binary associative operations satisfying certain mixed associativity conditions, i.e.,
\begin{align*}
    (a \vd b) \dv c \,&=\, a \vd (b \dv c)\,,\\[0.3cm]
    (a \vd b) \vd c = (a \dv b) \vd c \qquad&\text{and}\qquad
    a \dv (b \dv c) = a \dv (b \vd c) \,, 
\end{align*}
for all $a,b,c \in D$.
Moreover, a digroup is a disemigroup characterized by the existence of unilateral inverse elements.
Digroups were independently introduced in \cite{Fel04}, \cite{Liu03}, and found successful application in \cite{Kin07} to determine a partial solution to the so-called \emph{coquecigrue} problem, cf. \cite{Lod92}, \cite{Lod01}.
The structure of digroups was slightly generalized in \cite{Diaz16} to that of $g$-digroups and have been the object of study in some recent articles, for instance, \cite{Ong20}, \cite{Marin-Gaviria25}, \cite{Diaz19}, \cite{Rodriguez23}, \cite{Diaz21}, \cite{Rodriguez21}, \cite{Smith22}, \cite{ZuPiZu24}.

To systematically determine families of Drinfel'd twists in pairs with conjugation racks, we introduce a ``skew brace"-like structure having a $g$-digroup as underlying structure. A \emph{{\rm(}left{\rm)} di-skew brace} is a quadruple $(D,\vd,\dv,\circ)$ where $(D, \vd, \dv)$ is a $g$-digroup, $(D, \circ)$ is a right group, and the following left distributivity laws 
\begin{align*}
        a \circ (b \vd c) =
        a \circ b \vd a^{-1} \vd a \circ c
        \ \qquad\text{and}\ \qquad
        a \circ (b \dv c) =
        a \circ b \dv a^{-1} \dv a \circ c\,,
\end{align*}
as well as the identity
$$
(a \vd b) \circ c = (a \dv b) \circ c \,,
$$
hold, for all $a,b,c \in D$, where $a^{-1}$ denotes the inverse of $a$ with respect to the operations $\vd$, $\dv$.
Let us point out that, in our writing,  we name $(D,\vd,\dv)$ the \emph{additive} structure and $(D,\circ)$ the \emph{multiplicative} one.
Moreover, we implicitly assume that multiplication has higher precedence than the sums. 
An interesting aspect that emerges from our investigations is the fact that the solutions determined by di-skew braces are generally not equivalent to those coming from skew braces \cite{GuVe17}. From a much broader perspective, we show that di-skew brace solutions generally are not even Drinfel'd isomorphic to skew brace solutions.

\smallskip

In \cref{Sec1}, we give the necessary basics on $g$-digroups and focus on a structural description that is helpful throughout the paper. In addition, we recall some basics on solutions and on the role of shelves in the study of left non-degenerate solutions.
In this context, we provide a general method to compute the order of a bijective left non-degenerate solution that involves only the left multiplication operators related to their associated racks. 
In \cref{Sec2}, we introduce di-skew braces and focus on studying those structural properties that permit us to show how they determine bijective non-degenerate solutions. 
In particular, to prove the right non-degeneracy of di-skew brace solutions, we translate in our language a result provided in \cite{Ru19} and hence show that the so-called \emph{square map} (or \emph{diagonal map}) of such a solution is bijective.
To conclude this section, we observe that di-skew brace solutions are Drinfel'd twists of non-necessarily idempotent racks in the sense of \cite{DoRySte24}, contrary as to what happens for skew brace solutions that are Drinfel'd twists of conjugation quandles.
In \cref{Sec3}, we focus on classes of examples that naturally come from a notion of compatible averaging operators on groups.
Let us note that averaging operators first appeared in fluid dynamics and invariant theory under the name of \emph{Reynolds operator} and have been recently studied for their connections with racks and Rota--Baxter operators, see for example \cite{Bar24x}, \cite{ChMa25x}, \cite{Das24x} and \cite{ZhGa24x} and the references therein. 
In \cref{Sec4}, we show that di-skew brace solutions turn out to lie in the class of dynamical extensions of solutions as in \cite{CasCatSte22}, once one adopts a point of view purely in terms of their Drinfel'd twists and associated racks.
In this manner, we offer a perspective that sets a direct link between the dynamical extensions of self-distributive structures as in \cite{AndGr03} and left non-degenerate solutions.
Specifically, we prove that if $(D,\vd,\dv,\circ)$ is a di-skew brace and $r_D$ its associated solution, then there exists a skew brace structure $G$ such that $r_D$ is a hemi-semidirect product $r_D = r_G \,\text{\scalebox{0.9}{${\dv\vd}$}$_{\psi,\sigma}$}\, r_E$, where $E$ is the set of idempotents of $D$ and $(\psi,\sigma)$ is a hemi-semidirect pair on $G$ with coefficients in $E$. This result is strictly connected to a similar decomposition  $D = G \,\text{\scalebox{0.9}{$\dv\vd$}}\, E$ that holds for a $g$-digroup $D$ (as shown in \cref{Sec1}).
As a final application, we employ the methods above to obtain necessary and sufficient conditions ensuring that a di-skew brace solution $r_D$ has finite order and, in this case, we explicitly compute it. Based on these results, we obtain that di-skew braces solutions are generally not Drinfel'd isomorphic to skew brace solutions, a distinction that further motivates a future deeper exploration of the structures introduced here.

\smallskip

\section{Basics}\label{Sec1}
The present section serves to introduce basic notions on generalized digroups and to fix the notation that will be used throughout the text.
Moreover, we provide some basic facts on solutions of the set-theoretic Yang--Baxter equation with a focus on their relationship with self-distributive structures.
For the sake of completeness, let us note that the term ``digroup'' already appears in \cite{BoFaPo23} and other works connected to skew braces, although with a reference to a structure that is different from ours.

\smallskip

\subsection{Generalized digroups}

\begin{defin}[cf. \cite{Lod01}]\label{def:digroups}
    A \emph{disemigroup} is a set $D$ equipped with two associative binary operations $\vd$, $\dv$ : $D \times D \to D$ satisfying the following conditions:
    \begin{align}
        a \vd (b \dv c) &= (a \vd b) \dv c \,,\quad \text{(inner associativity)} \label{inn_assoc} \\
        a \dv (b \vd c) &= a \dv (b \dv c) \,,\quad \text{(right bar-side irrelevance)} \label{rbarside} \\
        (a \vd b) \vd c&= (a \dv b) \vd c \,,\quad \text{(left bar-side irrelevance)} \label{lbarside}
    \end{align}
    for all $a,b,c \in D$.
    A \emph{dimonoid} is a disemigroup $(D,\vd,\dv)$ admitting an element $e \in D$ called a \emph{bar-unit} that satisfies
    \begin{align}
        \forall\,a \in D \quad e \vd a = a = a \dv e \,.\label{bar-unit}
    \end{align}
    The set of all bar-units in $D$ is called the \emph{halo} of $D$ and is denoted as ${E(D,\vd,\dv)}$.
    The notion of \emph{subdisemigroups} applies in this context in a natural way; in particular, if $S$ is a subdisemigroup of a disemigroup $D$ then $E(S) = E(D) \cap S$.
\end{defin}

In a similar vein of \cite{Smith22}, we introduce the following definition.

\begin{defin}
    A \emph{generalized digroup} ($g$-digroup for simplicity) is a dimonoid $(D,\vd,\dv)$ such that for all $a \in D$ and all $e \in E(D)$ there exist  a unique pair of elements $I_e(a)$, $J_e(a) \in D$ satisfying the following identities:
    \begin{align}
        I_e(a) \dv a = e = a \vd J_e(a) \,.\label{def:inverse}
    \end{align}
    In what follows, for all $e \in E(D)$ and $a \in D$ it is useful to denote $I_e(a)$ as $a^{I_e}$ and $J_e(a)$ as $a^{J_e}$.
\end{defin}

A homomorphism of $g$-digroups is just a homomorphism of the underlying disemigroups.
It is worth to remark the fact that if $f:(D,\vd_1,\dv_1) \to (E,\vd_2,\dv_2)$ is a homomorphism of $g$-digroups, then $f(D,\vd_1,\dv_1)$ is a generalized subdigroup of $(E,\vd_2,\dv_2)$ and $f$ maps bar-units of $D$ to bar-units of $f(D)$.
Moreover, if $e \in E(D,\vd_1,\dv_1)$ and $a \in D$ then
\begin{align*}
        f(a^{I_e}) = f(a)^{I_{f(e)}} 
        \qquad \text{and}\qquad
        f(a^{J_e}) = f(a)^{J_{f(e)}}.
    \end{align*}

\begin{rem}\label{rem:dig-group}
    If $(D,\vd)$ is a group then $(D,\vd,\vd)$ is clearly a $g$-digroup.
    Conversely, if $(D,\vd,\dv)$ is a $g$-digroup such that $\vd\ =\ \dv$ then $(D,\vd) \ = \ (D,\dv)$ is a group.
\end{rem}

\begin{ex}\label{ex:digroup}
    Let $G$ be a group endowed with a right action $\psi:G \to \Sym(E)$ on a non-empty set $E$.
    On $G \times E$ define two binary operations $\vd,\dv$ by setting
    \begin{align*}
        (g,e) \vd (h,f) = (gh,f) 
        \,,\quad 
        (g,e) \dv (h,f) = (gh,\psi_h(e)) \,,
    \end{align*}
    for all $g,h \in G$, $e,f \in E$.
    Then, the structure $(G \times E,\vd,\dv)$ is a $g$-digroup.
\end{ex}

The example above can actually be promoted to a structural result that concerns the whole class of $g$-digroups.
Specifically, as already shown in \cite{Kin07} and \cite{Diaz16}, if $(D,\vd,\dv)$ is a $g$-digroup then there exists a group $G$ and a non-empty set $E$ equipped with a right action of $G$ such that $(D,\vd,\dv)$ is isomorphic to the $g$-digroup $G \times E$ as constructed above.

In what follows, we propose an intrinsic reformulation of this fact by reframing it in a vest that better suits our purposes.
In particular, we show that $(D,\vd)$ is a right group while $(D,\dv)$ is a left group satisfying a certain compatibility property.

\medskip

Let us recall that a semigroup $S$ is a \emph{right group} if for all $a,b \in S$ there exists a unique $x\in S$ such that $ax = b$. Equivalently, $S$ is a right group if all left multiplication operators $l_a:S\to S, x\mapsto ax$ are bijective, for all $a \in S$. Left groups are dual to right groups.
Moreover, from the general theory of semigroups we note that if $S$ is a right group, if $E(S)$ is the set of its idempotents and if we fix $e\in E(S)$, then $G = Se$ is a subgroup of $S$ having $e$ as unit and $S$ is a direct product of $G$ and $E(S)$, i.e. satisfying $G\cap E(S) = \{e\}$. We refer the reader to \cite{ClPr61} for more details. 

\begin{prop}[cf. \cite{Diaz16}]\label{prop:gdigroup_properties}
    Let $(D,\vd,\dv)$ be a $g$-digroup. Then, the following identities hold:
    \begin{enumerate}
        \item[1)] 
        $(a^{I_e})^{I_e} =
            e \dv a =
            (a^{J_e})^{I_e}$ 
            \ and \ 
             $(a^{J_e})^{J_e} =
            a \vd e =
            (a^{I_e})^{J_e}$;
        \item[2)] 
        $a^{I_e} \vd b =
            a^{J_\xi} \vd b  
        $ \ and \ $a \dv b^{I_e} =
            a \dv b^{J_\xi} 
        $;
        \item[3)] 
        $(a \vd b)^{I_e} = b^{I_e} \dv a^{I_e} =
            (a \dv b)^{I_e}$ 
            \ and \
            $(a \vd b)^{J_e} =
            b^{J_e} \vd a^{J_e} =
            (a \dv b)^{J_e} \,,$
    \end{enumerate}
    for all $a,b \in D$ and $e, \xi \in E(D)$.
\end{prop}

\begin{rem}\label{rem:inverses}
    Thanks to $2)$ in \cref{prop:gdigroup_properties}, the action of unilateral inverses on the bar side of expressions such as $a^{I_e} \vd b$ or $a \dv b^{J_f}$, with $a,b \in D$, does not depend on the specific bar-unit or side chosen.
    Therefore, in what follows we will write $a^{-1} \vd b$ and $a \dv b^{-1}$ to denote these expressions, respectively, except when a choice is needed in order to accomplish a computation.
\end{rem}

\begin{prop}\label{g-digroup_decomp}
    Let $(D,\vd,\dv)$ be a $g$-digroup.
    Then the following hold:
    \begin{enumerate}
        \item[1)] $(D,\,\vd)$ is a right group and $(D,\dv)$ is a left group.
        \item[2)] The bar-units are the only idempotents, i.e.,
        $E(D,\,\vd) = E(D) = E(D,\,\dv)$.
    \end{enumerate}
    \begin{proof}
        Let us first prove that $(D,\vd)$ is a right group.
        Fix $a,b \in D$, choose a bar-unit $e \in E(D)$ and set $x = a^{I_e} \vd b$.
        Then, it follows that $a \vd x = (a \vd a^{I_e}) \vd b = b$.
        Moreover, if $x' \in D$ is another element such that $a \vd x' = b$, then
        \[
            x' = e \vd x' = (a^{I_e} \vd a) \vd x' = a^{I_e} \vd b = x \,,
        \]
        thus proving the assertion.
        A dual line of reasoning shows that $(D,\dv)$ is a left group.

        To prove the second point, notice first that the bar-units of $D$ are already idempotent both in $(D,\vd)$ and $(D,\dv)$.
        Moreover, if $e \in E(D,\vd)$ and $a \in D$, then from the identity ${e \vd a = e \vd (e \vd a)}$ it follows that $e \vd a = a$, thanks to left cancellativity of $(D,\vd)$, thus proving that $E(D) = E(D,\vd)$.
        In an analogous manner one can show that ${E(D) = E(D,\dv)}$.
    \end{proof}
\end{prop}

Let $(D,\vd,\dv)$ be a $g$-digroup and fix a bar-unit $\xi \in D$.
If we set $G = D \vd \xi$ and $H = \xi \dv D$, then \cref{g-digroup_decomp} ensures that $G$ is a subgroup of $(D,\vd)$ while $H$ is a subgroup of $(D,\dv)$, both admitting $\xi$ as neutral element.
Moreover, for all elements $a \in D$ there exist  a unique pair $(g_a,h_a) \in G \times H$ such that the following decompostion hold:
\begin{align*}
    g_a \vd e_a = a = f_a \dv h_a \,.
\end{align*}
The two decompositions above are intertwined in a way that is described in the proposition below.
In order to distinguish one from the other, hereinafter we will call $g_a$ (resp. $h_a$) the left (resp. right) groupal component of $a$ and $e_a$ (resp. $f_a$) the left (resp. right) idempotent component of $a$.

\begin{prop}
    Let $(D,\vd,\dv)$ be a $g$-digroup, fix $\,\xi \in E(D)$ and consider the subgroups $G = D \vd \xi$ and $H = \xi \dv D$ of $(D,\vd)$ and $(D,\dv)$, respectively.
    If an element $a \in D$ admits the decomposition $g_a\vd e_a = a = f_a \dv h_a$ into its groupal and idempotent components, then the following holds:
    \begin{align*}
        \begin{cases}
            g_a = h_a \vd \xi \\
            e_a = h_a^{-1} \vd f_a \dv h_a
        \end{cases}
        \quad\text{and}\qquad
        \begin{cases}
            h_a = \xi \dv g_a \\
            f_a = g_a \vd e_a \dv g_a^{-1}
        \end{cases}
    \end{align*}
\end{prop}

\smallskip 

\begin{rem}
    As a useful final complement let us note that if $(D,\vd,\dv)$ is a $g$-digroup then $(D,\vd)$ and $(D,\dv)$ are anti-isomorphic semigroups.
    Indeed, the map $\psi:D \to D$ defined by setting $\psi(a) = g_a^{-1} \vd f_a$, for all $a \in D$, is a semigroup anti-isomorphism. 
\end{rem}

\smallskip

\subsection{Basics on set-theoretic solutions}

Among the various algebraic structures that originated from or were employed in the study of set-theoretical solutions, the class of \emph{shelves} has a key role in their investigation and are useful to encode different properties.

\begin{defin}
    A (left) shelf $(D, \triangleright)$ is a set $D$ equipped with a binary operation $\triangleright$ such that
    \begin{equation}\label{self_distri}
        \forall\, a,b,c\in D\quad a \triangleright (b \triangleright c)= (a \triangleright b) \triangleright (a \triangleright c)\,. 
    \end{equation} 
    A shelf whose left multiplication operators ${L_a:D \to D\,, b \mapsto a \tr b}$ are all bijective is called a \emph{rack}. If in addition $a\triangleright a = a$ holds, for all $a\in D$, then $D$ is called a \emph{quandle}.
\end{defin}

{Hereinafter, when considering a shelf $(D,\tr)$ we will always denote with $L_a:D \to D$ the operator of left multiplication by $a \in D$.
With this notation, it is important to note that the identity \eqref{self_distri} is equivalent to the following
\begin{align}\label{self_distri1}
    \forall\, a,b \in D \quad L_aL_b = L_{L_a(b)}L_a \,,
\end{align}
which on its own amounts to the fact that $L_a \in \End(D,\tr)$, for all $a \in D$.}

\begin{exs}\,
    \begin{enumerate}
        \item[1)] If $D$ is a set and $f:D\to D$ a map, the binary operations on $D$ defined by setting $a \tr b = f(b)$, for all $a,b\in D$, endows $D$ of a shelf structure that is a rack if and only if $f \in \Sym(D)$. In particular, if $f=\id_D$, then we call $\left(D,\triangleright\right)$ the \emph{trivial quandle}.
        \item[2)] If $D$ is a group, then the structure $\left(D,\triangleright\right)$ with $a\triangleright b:= a^{-1}ba$, for all $a,b\in D$, is a quandle called the \emph{conjugation quandle on $D$} and denoted by $\Conj(D)$.
    \end{enumerate}
\end{exs}

Let $\left(D,\,\tr\right)$ be a shelf and let us consider the map $r_{\tr}:D \times D \to D \times D$ defined by setting $r_{\tr}(a,b) = \left(b,\, b\tr a\right)$, for all $a,b \in D$.
Then, $r_\tr$ is a left non-degenerate solution that is said to be of \emph{(left) derived type}.
It is clear that $r_{\tr}$ is bijective and non-degenerate if and only if $(D,\tr)$ is a rack.
Conversely, as shown for example in \cite{Le18}, to a left non-degenerate solution $\left(D, r\right)$ one can always associate a shelf $\left(D,\triangleright_r\right)$ defined by setting
\begin{align*}
    \forall\, a,b\in D\qquad a \tr_r b = \lambda_a\rho_{\lambda_b^{-1}(a)}(b) \,.
\end{align*}
{An important property that relates the latter structure with the original solution is the fact that each map $\lambda_a:D \to D$ is an automorphism of $(D,\tr_r)$, i.e. such that
\begin{align}\label{lambda_auto}
    \forall\, a,b,c\in D\qquad \lambda_a(b \tr_r c) = \lambda_a(b) \tr_r \lambda_a(c)
\end{align}
(see also \cref{thm:twist_solution}). Moreover}, the solution further determined by $\tr_r$ takes the name of \emph{(left) derived solution} of $r$ and is usually denoted as $r'$.
It is worth to note that $(r')' = r$.

\smallskip

Let $(D,r)$ and $(E,s)$ be solutions.
A \emph{Drinfel'd homomorphism} ($D$-homomorphism for short) $\varphi:(D,r) \to (E,s)$ is a map $\varphi:D \to E$ that satisfies $s \varphi = \varphi r$.
If $\varphi$ occurs to be bijective then it is said to be a $D$-isomorphism, while the solutions $(D,r)$ and $(E,s)$ are called $D$-isomorphic.
For more details, we refer the reader for instance to \cite{Do21}, \cite{KuMu00}, and \cite{So00}.
Particular examples of such maps $\varphi:(D,r) \to (E,s)$ are given by those that can be decomposed as $\varphi = f \times f$, for some  map $f:D \to E$ which bears the name of \emph{homomorphism} of solutions.
When $f$ happens to be bijective, we say that $f$ is an \emph{equivalence} and call $(D,r)$ and $(E,s)$ \emph{equivalent} (or isomorphic) solutions, cf. \cite{ESS99}.

\begin{rem}\label{rem:shelf_hom}
    It is worth noting here that if $(D,r)$ and $(E,s)$ are left non-degenerate solutions, then an equivalence $f:(D,r) \to (E,s)$ corresponds to an isomorphism between their respective left derived shelves $f:(D,\tr_r) \to (E,\tr_s)$.
\end{rem}

\begin{prop}\label{prop:derived_sol}
Let $\left(D, r\right)$ be a left non-degenerate solution and let $\left(D, r'\right)$ be its left derived solution.
Then, $\left(D, r\right)$ and $\left(D, r'\right)$ are $D$-isomorphic.
\begin{proof}
    The assertion is the content of \cite[Lemma~2.12]{DoRySte24}.
\end{proof}
\end{prop}

The above relationship can actually be lifted to a fundamental description of all left non-degenerate solutions through the so-called notion of \emph{twist}. 
In detail, following \cite[Definition 2.14]{DoRySte24}, if $(D,\tr)$ is a shelf then a map $\varphi:D \to \Aut(D,\tr)$ is called a {twist} of $(D,\tr)$ if the following identity holds:
\begin{align}\label{eq:twist}
    \forall\, a,b\in D\qquad
    \varphi_a\varphi_b = 
    \varphi_{\varphi_a(b)}\varphi_{\varphi^{-1}_{\varphi_a(b)}\left(\varphi_a(b)\, \tr\,a\right)} \,
\end{align}

\begin{theor}[cf. \cite{DoRySte24}]\label{thm:twist_solution}
    Let $(D,\tr)$ be a shelf and consider a map $\varphi:D \to {\Sym(D)}$.
    If we define $r:D \times D \to D \times D$ by setting
    \begin{align*}
        r(a,b) =
        \left(\varphi_a(b), \varphi^{-1}_{\varphi_a(b)}\left(\varphi_a(b) \tr a\right)\right) \,,
    \end{align*}
    for all $a,b \in D$, then $(D,r)$ is a solution if and only if $\varphi$ is a twist.
    In this case, $(D,r)$ is left non-degenerate and $(D,r_\tr)$ is its left derived solution.
    {Moreover, any left non-degenerate solution can be obtained that way.}
\end{theor}

In the setting of \cref{thm:twist_solution}, if $(A,r)$ and $(B,s)$ are left non-degenerate solutions, then an equivalence $f:A\to B$ corresponds to an automorphism $f:(A,\tr_r) \to (B,\tr_s)$ of shelves such that $f\lambda_a(b) = \lambda_{f(a)}f(b)$ holds, for all $a,b\in A$.

\medskip

We conclude this section by showing how the description given so far can be used to obtain two important general results that will be specifically useful in \cref{SubSec:Sub-sol-di-skew_brace}.
\smallskip

The first result is related to a characterization of right non-degeneracy provided by Rump in \cite[Corollary~2, p.~21]{Ru19}  using the language of \emph{$q$-cycle sets}. For completeness, we propose to report only the part of Rump's result that we use together with the related proof, but in a form that better suits our notation.
With this aim, let us first recall that if $(D,r)$ is a left non-degenerate solution then its \emph{square map} (or \emph{diagonal map}) $\mathfrak{q}:D \to D$ is defined by $\mathfrak{q}(a) = \lambda^{-1}_a(a)$, for all $a \in D$.

\begin{prop}\label{lem:rho_nondegenerate_sol}\label{prop:bijleft-right}
    Let $(D,r)$ be a bijective left non-degenerate solution.
    Then $r$ is right non-degenerate if and only if its {square map} is bijective.
    \begin{proof}
        {Let $(D,\tr_r)$ be the left derived rack of $(D,r)$.}
        Fix $b \in D$ and define a map $\theta_b:D\to D$ by setting $\theta_b(a):= \lambda_b^{-1}L_b\mathfrak{q}\left(a\right)$, for all $a\in D$. Then, for all $a \in D$, the following holds: 
        \begin{align*}
            \theta_b\left(a\right)
            =  \lambda^{-1}_b\lambda^{-1}_a L_{\lambda_a\left(b\right)}\left(a\right)
            = \lambda^{-1}_{\rho_b\left(a\right)}\lambda^{-1}_{\lambda_a\left(b\right)}L_{\lambda_a\left(b\right)}\left(a\right) 
            = \lambda^{-1}_{\rho_b\left(a\right)}\rho_b\left(a\right) 
            =  \mathfrak{q}\rho_b\left(a\right),
        \end{align*}
        where we have used {\eqref{lambda_auto} in the first equality and} \eqref{first} in the second equality.
        We deduce that $\rho_b$ is bijective, for all $b\in D$, if $\mathfrak{q}$ is bijective.
        Conversely, the necessity of this last assumption is proven in \cite[Lemma 3.1]{LeVe19} (non-necessarily for finite solutions, see \cite[Remark 3.10]{LeVe19}).
    \end{proof}
\end{prop}

\begin{rem}
    Let us observe that \cref{lem:rho_nondegenerate_sol} admits a straightforward dual statement which can be proven in analogous manner, ensuring that a bijective right non-degenerate solution $(D,r)$ is left non-degenerate if and only if the map  
    $\bar{\mathfrak{q}}:D\to D$ defined by setting $\bar{\mathfrak{q}}(a) = \rho^{-1}_a(a)$, for every $a\in D$, is bijective.
\end{rem}

\smallskip

The last part of this section is based on the elementary fact that if $(D,r)$ is the solution associated to a twist $\lambda:D \to \Aut(D,\tr)$, then $r$ and $r'$ share the same order as elements of $\Sym(D \times D)$, thanks to \cref{prop:derived_sol}.
This observation opens up the possibility to relate the above quantity to some properties of the associated shelf $(D,\tr)$ and compute it when possible.
\medskip

\begin{lemma}\label{powers_prop}
    Let $(D,\triangleright)$ be a rack. Then, the following identities 
    \begin{align}
        r_\tr^{2n}(a,b) &= \left((L_bL_a)^nL_a^{-n}(a),(L_bL_a)^nL_b^{-n}(b)\right) \label{even_powers2}\,, \\[0.2cm]
        r_\tr^{2n+1}(a,b) &= \left((L_bL_a)^nL_b^{-n}(b),(L_bL_a)^{n+1}L_a^{-(n+1)}(a)\right) \label{odd_powers2}\,,
    \end{align}
    hold, for all $a,b \in D$ and $n\in\mathbb{N}$.
\end{lemma}
\begin{proof}
    As a technical identity, let us first note that, {by \eqref{self_distri1}},
    \begin{align}\label{eq:tech_id}
        L_{\left(L_xL_y\right)^kL_u^{-k}(v)} =
        \left(L_xL_y\right)^kL_u^{-k}L_vL_u^k\left(L_xL_y\right)^{-k}
    \end{align}
    holds for all $x,y,u,v \in D$ and $k \in \mathbb{N}$.
    To prove \eqref{even_powers2} we proceed by induction on $n \in \mathbb{N}$.
    In the base case $n = 1$ we have that 
    $r^2_\tr(a,b) = \left(L_b(a),L_{L_b(a)}(b)\right) = \left(L_b(a),L_bL_aL_b^{-1}(b)\right)$, for all $a,b \in D$.
    Now, let $n \in \mathbb{N}$ with $n > 1$ and assume the assertion true for $n > 1$.
    Fix $a,b \in D$ and note that
    \begin{align*}
        r_\tr^{2(n+1)}(a,b) 
            = 
            r_\tr^2 r_\tr^{2n}(a,b)
            =
            (L_B(A),L_{L_B(A)}(B)) \,,
    \end{align*}
    where we have set $A = (L_bL_a)^{n}L_a^{-n}(a)$ and $B = (L_bL_a)^{n}L_b^{-n}(b)$.
    Applying \eqref{eq:tech_id}, we compute $L_B(A) = \left(L_bL_a\right)^{n+1}L_a^{-n-1}(a)$ and again by \eqref{eq:tech_id} deduce that
    \begin{align*}
        L_{L_B(A)}(B) 
        = \left(L_bL_a\right)^{n+1}L_a\left(L_bL_a\right)^{-n-1}(B) 
        = \left(L_bL_a\right)^{n+1}L_b^{-n-1}(b) \,,
    \end{align*}
    thus proving the assertion.
    To verify that \eqref{odd_powers2} holds, for all $n \in \mathbb{N}$, it is sufficient to compute $r_\tr^{2n+1} = r^{}_\tr r_\tr^{2n}$ using \eqref{even_powers2} and \eqref{eq:tech_id} accordingly.
\end{proof}

\smallskip

\begin{prop}\label{prop-even-odd}
    Let $(D,\tr)$ be a rack and let $n \in \mathbb{N}$.
    Then, the following hold:
    \begin{enumerate}
        \item[\textnormal{1)}] $r_\tr^{2n} = \id_D$ if and only if 
        \begin{align}\label{eq:1} 
        \forall\ a,b\in D\qquad (L_bL_a)^nL_b^{-n}(b) = b \,.
         \end{align}
         \item[\textnormal{2)}] $r_\tr^{2n+1} = \id_D$ if and only if $(D,\tr)$ is a quandle  such that
    \begin{align}\label{eq:2}        
    \forall\ a,b\in D\qquad (L_aL_b)^n(a) = b \,.
\end{align}
    \end{enumerate}
\end{prop}
\begin{proof}
    Notice first that if $r_\tr^{2n} = \id_D$ then $\left(L_bL_a\right)^nL_b^{-n}(b) = b$ trivially holds for all $a,b \in D$, thanks to \eqref{even_powers2}.
    Conversely, suppose that \eqref{eq:1} identically holds and let us show that $(L_bL_a)^nL_a^{-n}(a) = a$, for all $a,b \in D$, thus proving the first point again thanks to \eqref{even_powers2}.
    Indeed, if $a,b \in D$ then
    \begin{align*}
        \left(L_bL_a\right)^nL_a^{-n}(a) 
        &\overeq{\eqref{eq:1}}
        \left(L_bL_a\right)^n\left(L_aL_b\right)^{-n}(a)
        =
        L_a^{-1}\left(L_aL_b\right)^nL_a\left(L_aL_b\right)^{-n}(a)\\
        &\overeq{\eqref{self_distri1}}
        L_a^{-1}L_{\left(L_aL_b\right)^nL_a^{-n}(a)}(a) = L_a^{-1}L_a(a) = a \,.
    \end{align*}
    To obtain the second point, it is sufficient to show that $r_\tr^{2n+1}(a,b) = (a,\,L_b(b))$, for all $a,b\in D$, under the assumption that $(L_bL_a)^nL_b^{-n}(b) = a$ holds, for all $a,b\in D$. 
    Indeed, if we observe that {the latter identity yields}
        \begin{align*}
            (L_bL_a)^{n+1}L_a^{-(n+1)}(a) 
            &= (L_bL_a)^n L_bL_a^{-n}(a) \\
            &= {(L_bL_a)^n L_b (L_aL_b)^{-n}(b)} \\
            &= L_b (L_aL_b)^n(L_aL_b)^{-n}(b) \\
            &= L_b(b) \,,
        \end{align*}
    for all $a,b \in D$, then we deduce that the above assertion holds thanks to \eqref{odd_powers2}, thus terminating the proof. 
\end{proof}
\smallskip

In the following result, we will implicitly assume that $\textnormal{inf}(\emptyset) = \infty$.

\begin{theor}\label{theor:derived_sol_order}
    Let $(D,\tr)$ be a rack with $|D| > 1$. 
    If we set
    \begin{align*}
        M_D &:= \textnormal{inf}\{n \in \mathbb{N} \mid \forall\, a,b\in D \quad (L_bL_a)^nL_b^{-n}(b) = a\} \\[0.2cm]
        N_D &:= \textnormal{inf}\{n \in \mathbb{N} \mid \forall\, a,b\in D \quad (L_bL_a)^nL_b^{-n}(b) = b\}
    \end{align*}
    and denote by $o(r_\tr)$ the order of $r_\tr$ as an element of $\Sym(D \times D)$, then the following assertions hold:
    \begin{enumerate}
        \item[\textnormal{1)}] If $D$ is not a quandle, then $r_\tr$ has finite order if and only if $N_D$ is finite.
        In particular, we have that $o(r_\tr) = 2N_D$.
        \item[\textnormal{2)}] If $D$ is a quandle, then $r_\tr$ has finite order if and only if\, $\textnormal{min}\{M_D,\,N_D\}$ is finite.
        In particular, the following cases occur:
        \begin{itemize}
            \item[\textnormal{(a)}] If $M_D < N_D$ then $o(r_\tr) = 2M_D+1$;
            \item[\textnormal{(b)}] 
            If $N_D < M_D$ then $o(r_\tr) = 2N_D$.
        \end{itemize}
    \end{enumerate}
    \begin{proof}
        The proof is a straightforward consequence of \cref{prop-even-odd}.
    \end{proof}
\end{theor}

\smallskip

As an application, we provide an example extending \cite[Theorem~4.13]{SmVe18} in the sense that the involved twists are not necessarily obtained through skew braces.

\begin{cor}\label{corol:conj_sol_order}
    Let $G$ be a group with centre $Z(G)$ and let $(G,\tr) = \Conj(G)$ be its conjugation quandle.
    If $r_\tr$ is its associated left derived solution, then $r_\tr$ has finite order if and only if $G/Z(G)$ is periodic.
    In particular, when this happens we have that
    \begin{align*}
        o(r_\tr) = 2\exp(G/Z(G)) \,.
    \end{align*}
    \begin{proof}
        Let us first note that we can suppose $G$ non-trivial without loss of generality.
        In the notation of \cref{theor:derived_sol_order}, if we assume that there exists $n \in \mathbb{N}$ such that $(L_bL_a)^nL_b^{-n}(b) = a$, for all $a,b \in G$, then we can deduce that $(ab)^{-n} b (ab)^n = a$, for all $a,b \in G$.
        Setting $b = 1$, it follows that $a = 1$, for all $a \in G$, a contradiction.
        Therefore, thanks to $\textnormal{2})$ in \cref{theor:derived_sol_order} we deduce that $r_\tr$ has finite order if and only if there exists $n \in \mathbb{N}$ such that $(ab)^{-n} b (ab)^n = b$, for all $a,b \in G$.
        Since the latter condition is equivalent to requiring that there exists $n \in \mathbb{N}$ such that $x^n y = y x^n$, for all $x,y \in G$, we conclude that $o(r_\tr)$ is finite if and only if $G/Z(G)$ is periodic and in this case $o(r_\tr) = 2\exp(G/Z(G))$.
    \end{proof}
\end{cor}

The following result can be proven through a similar reasoning.

\begin{cor}
    Let $G$ be a group and let $\text{Core}(G) = (G,\tr)$ be the core quandle on $G$, whose operation is defined by $a \tr b = a b^{-1} a$, for all $a,b \in D$.
    Then $r_\tr$ has finite order if and only if $G$ is periodic.
    In particular, when this happens we have that $o(r_\tr) = \exp(G)$.
\end{cor}

\smallskip

\section{Di-skew braces}\label{Sec2}         

\begin{defin}\label{def:di_skewbrace}
    A (left) \emph{di-skew brace} is a quadruple $(D,\vd,\dv,\circ)$ where $(D,\vd,\dv)$ is a $g$-digroup, $(D,\circ)$ is a right group such that the following identities hold for all $a,b,c \in D$:
    \begin{align}
        a \circ (b \vd c) &=
        a \circ b \vd a^{-1} \vd a \circ c \,, \tag{D1}\label{diskew1} \\
        a \circ (b \dv c) &=
        a \circ b \dv a^{-1} \dv a \circ c \,, \tag{D2}\label{diskew2} \\
        (a \vd b) \circ c &=
        (a \dv b) \circ c \,. \tag{D3}\label{diskew3}
    \end{align} 
    We call $(D,\vd,\dv)$ and $(D,\circ)$ the \emph{additive structure} and the \emph{multiplicative structure} of $(D,\vd,\dv,\circ)$, respectively. Moreover, we say that the structure $(D,\vd,\dv,\circ)$ is a \emph{di-brace} if $(D,\vd,\dv)$ is \emph{abelian}, i.e., if $a \vd b = b \dv a$ holds for all $a,b \in D$.
\end{defin}

Note that the class of di-skew braces encompasses that of skew braces $(D,+,\circ)$ \cite{GuVe17}, where we interpret their additive structure as an iterated group $(D,+,+)$.

\begin{ex}
    Let $(D,\vd,\dv)$ be a $g$-digroup.
    Then $(D,\vd,\dv,\vd)$ is a di-skew brace that will be called the \emph{trivial di-skew brace}.
    Analogously, if we denote with $\dv^{\text{op}}$ the opposite operation of $\dv$, then the quadruple $(D,\vd,\dv,\dv^{\text{op}})$ is again a di-skew brace which bears the name of \emph{almost-trivial di-skew brace}.
\end{ex}

We postpone further examples of di-skew braces in Section 3, where we employ an extended notion of averaging operator on a group. 

\smallskip

In the following, we will denote with $E(D,\vd,\dv)$ the set of bar-units of $(D,\vd,\dv)$, while $E(D,\circ)$ will stand for the set of idempotents in $(D,\circ)$.

\begin{lemma}\label{common_idemp}
    Let $(D,\vd,\dv,\circ)$ be a di-skew brace.
    Then $E(D,\vd,\dv) = E(D,\circ)$.
    \begin{proof}
        Let us first prove that $E(D,\circ) \subseteq E(D,\vd,\dv)$.
        If $j \in E(D,\circ)$ then from the fact that $j$ is a left unit in $(D,\circ)$, it follows that
        \begin{align*}
            j \vd j = j \circ (j \vd j) 
            = j \circ j \vd j^{-1} \vd j \circ j 
            = j \vd j^{-1} \vd j 
            = j \,,
        \end{align*}
        where in the last equality we have used the fact that $j \vd j^{-1} \in E(D,\,\vd,\dv)$.
        Hence, ${j \in E(D,\vd,\dv)}$ thanks to \cref{g-digroup_decomp}.
        To prove the converse inclusion, consider a bar-unit $e \in E(D,\vd,\dv)$, fix $j \in E(D,\circ)$ and notice that 
        \begin{align*}
            e \circ a = (j \vd e) \circ a \overeq{\eqref{diskew3}} (j \dv e) \circ a = j \circ a = a \,,
        \end{align*}
        holds for all $a \in D$.
        Therefore $e \in E(D,\circ)$, concluding the proof.
    \end{proof}
\end{lemma}

\begin{convention}\label{convenzione}
    \textnormal{In light of the previous results, in what follows, if $(D,\vd,\dv,\circ)$ is a di-skew brace then we will denote by $E(D)$ its set of common idempotents which we call the set \emph{idempotents} of $(D,\vd,\dv,\circ)$.
    Most importantly, throughout our work and if not otherwise stated, we will implicitly consider a fixed idempotent $0 \in E(D)$ which will be employed to decompose both the additive $(D,\vd,\dv)$ and multiplicative $(D,\circ)$ structures.
    In particular, since $(D,\circ)$ is a right group, then $M = D \circ \xi$ is a subgroup of $(D,\circ)$ and all $a \in D$ admit a unique decomposition as 
    $$
    a = m_a \circ u_a \,,
    $$
    where $m_a\in M$ and $u_a\in E(D)$ are its groupal and idempotent component, respectively.}
\end{convention}

\begin{rem}
    It is worth remarking that if $(D,\vd,\dv,\circ)$ is a di-skew brace then the groupal parts $G$ and $M$ of the additive and multiplicative structures, respectively, do not necessarily coincide.
    Indeed, first note that if $a\in D$ satisfies $a\in G \cap M$, then $a \circ 0 = a \vd 0$.
    As a consequence, if in particular $(D,\vd,\dv)$ is a non-abelian $g$-digroup then the almost-trivial di-skew brace $(D,\vd,\dv,\dv^{\text{op}})$ provides an instance in which $G \neq M$.
    Analogously, with a dual reasoning, it is straightforward to verify that if $(D,\vd,\dv)$ is non-abelian, then the trivial di-skew brace $(D,\vd,\dv,\vd)$ satisfies $H \neq M$, where $H$ is the right groupal part of $(D,\vd,\dv)$.
\end{rem}

In what follows, if $(D,\vd,\dv,\circ)$ is a di-skew brace then, for all $a,b \in D$, we will denote with $a^- \circ b$ the unique $x \in D$ such that $a \circ x = b$.
{The slight abuse of notation is justified by the fact that in this way we can denote left division by $a$ with respect to $\circ$ - which is possible since $(D,\circ)$ is in particular a left quasigroup - without having to introduce a left division symbol in the signature of a di-skew brace.}
In particular, in light of \cref{convenzione}, if $a = m_a \circ u_a$ is the multiplicative decomposition of $a$ then $a^- \circ b = {m_a^{-}} \circ b$, for all $a,b \in D$, {thus avoiding any risk of ambiguity.}

\smallskip

\begin{prop}\label{digroup_lambda}
    Let $D$ be a di-skew brace and for all $a \in A$ define a map $\lambda_a:D \to D$ by setting $\lambda_a(b) = a^{-1} \vd a \circ b$, for all $b \in D$.
    Then, the following hold:
    \begin{enumerate}
        \item[1)] $\lambda_a \in \Aut(D,\vd,\dv)$ with inverse $\lambda_a^{-1}(b) = a^- \circ (a \vd b)$, for all $a,b \in D$;
        \item[2)] The map $\lambda:(D,\circ) \to \Aut(D,\vd,\dv)\,, a \mapsto \lambda_a$ is a semigroup homomorphism.
    \end{enumerate}
    \begin{proof}
        Fix $a \in D$ and define a map $\sigma_a:D \to D$ by setting $\sigma_a(b) = a^- \circ (a \vd b)$, for all $b \in D$.
        Then, a straightforward computation shows that $\sigma_a = \lambda_a^{-1}$.
        Now, let $b,c \in D$ and notice that
        \begin{align*}
            \lambda_a(b \vd c) =
            a^{-1} \vd a \circ (b \vd c) 
            = a^{-1} \vd (a \circ b \vd a^{-1} \vd a \circ c) 
            = \lambda_a(b) \vd \lambda_a(c) \,,
        \end{align*}
        while we also have that
        \begin{align*}
            \lambda_a(b \dv c) 
            = a^{-1} \vd (a \circ b \dv a^{-1} \dv a \circ c) 
            \overeq{\eqref{inn_assoc}} (a^{-1} \vd a \circ b) \dv (a^{-1} \dv a \circ c) 
            \overeq{\eqref{rbarside}} \lambda_a(b) \dv \lambda_a(c)\,,
        \end{align*}
        thus showing that $\lambda_a \in \Aut(D,\vd,\dv)$.
        To prove the second point, fix a bar-unit $e\in E(D)$ and notice that 
        \begin{align*}
            \lambda_{a \circ b}(c) &=
            (a \circ b)^{-1} \vd a \circ b \circ c \\
            &= (a \vd \lambda_a(b))^{-1} \vd a \circ b \circ c \\
            &= (\lambda_a(b))^{-1} \vd a^{-1} \vd a \circ b \circ c\,, \quad
            \text{since $\lambda_a$ is a homomorphism,} \\
            &= \lambda_a(b^{I_e}) \vd a^{-1} \vd a \circ b \circ c \\
            &= \lambda_a(b^{I_e}) \vd \lambda_a(b \circ c) \\
            &= \lambda_a(b^{-1} \vd b \circ c) \\
            &= \lambda_a\lambda_b(c) \,,
        \end{align*}
        for all $c \in D$, where the choice of the bar-unit is only necessary in the scope of the computation and does not affect the general result.
    \end{proof}
\end{prop}

\begin{rem}\label{rem:lambda_index}
    Let $(D,\vd,\dv,\circ)$ be a di-skew brace and consider the map ${\lambda:D \to \Aut(D,\vd,\dv)}$ from \cref{digroup_lambda}.
    Note that \eqref{diskew3} is equivalent to the identity $\lambda_{a \vd b} = \lambda_{a \dv b}$.
    Indeed, if $a,b,c, \in D$ then the following hold:
    \begin{align*}
        \lambda_{a \vd b}(c) = \lambda_{a \dv b}(c) 
        &\iff
        (a \vd b)^{-1} \vd (a \vd b) \circ c = (a \dv b)^{-1} \vd (a \dv b) \circ c \\
        &\iff
        b^{-1} \vd a^{-1} \vd (a \vd b) \circ c = b^{-1} \vd a^{-1} \vd (a \dv b) \circ c \\
        &\iff (a \vd b) \circ c = (a \dv b) \circ c \,,
    \end{align*}
    by left cancellativity of $(D,\vd)$.
\end{rem}

\begin{rem}\label{rem:diskew_decomp}
    Let $(D, \vd, \dv, \circ)$ be a di-skew brace and let $G$ be the groupal component in the decomposition of $(D,\vd,\dv)$.
    If $a,b \in D$, then the following holds:
    \begin{align*}
        \lambda_a(b) &\overeq{\eqref{rem:lambda_index}}
        \lambda_{g_a \dv e_a}(g_b\vd e_b) \overeq{\eqref{digroup_lambda}}
        \lambda_{g_a}(g_b) \vd \lambda_{g_a}(e_b) \,,
    \end{align*}
    where $\lambda_{g_a}(e_b) \in E(D)$ is a bar-unit, thanks to $1)$ in \cref{digroup_lambda}.
    Since $\lambda_{g_a}(g_b)$ is not necessarily in $G$, in general, for all $x,y \in G$ it makes sense to define $\sigma_x(y) := \lambda_x(y) \vd 0$ as the groupal component of $\lambda_x(y)$.
    As a consequence, the following identity
    \begin{align*}
        \lambda_a(b) = \sigma_{g_a}(g_b) \vd \lambda_{g_a}(e_b) \,,
    \end{align*}
    provides the additive decomposition of the twist $\lambda$, for all $a,b \in D$.
    In particular, as a useful reminder for \cref{theor:diskew_twist_decomp}, we note that if $a,b \in G$ then $a \circ b = a \vd \sigma_a(b) \vd \lambda_a(0)$.
\end{rem}

\smallskip

\subsection{Set-theoretic solutions associated to di-skew braces}\label{SubSec:Sub-sol-di-skew_brace}

In this section we recall the notion of conjugation rack associated to a $g$-digroup and show that di-skew braces systematically yield bijective and non-degenerate solutions having such a structure as left derived rack.
The conjugation rack was first introduced in \cite{Kin07} as an extension to digroups of the usual conjugation quandle of a group and was later carried over to $g$-digroups, see \cite{Res24}.

\begin{prop}
    Let $(D,\vd,\dv)$ be a $g$-digroup and define a binary operation $\triangleright$ on $D$ by
    \begin{align*}
        a \triangleright b = a^{-1} \vd b \dv a \,,
    \end{align*}
    for all $a,b \in D$.
    Then $(D,\triangleright)$ is a \emph{left rack} called the \emph{conjugation rack} of $D$ and denoted as $\Conj(D,\vd,\dv)$.      
\end{prop}

As already clear in \cite{Res24}, conjugation racks admit a structural result that will come useful in later considerations.

\begin{lemma}\label{lemma:conj_rack_decomp}
    Let $(D,\vd,\dv)$ be a $g$-digroup and let $G$ be its groupal component.
    Then, the map $\psi:G \to \Sym(E(D))$ defined by setting $\psi_g(e) := g^{-1} \vd e \dv g$, for all $g \in G$ and $e \in E(D)$, determines a right group action such that
    \begin{align*}
        a \tr b = (g_a^{-1} \vd g_b \dv g_a) \vd \psi_{g_a}(e_b) \,,
    \end{align*}
    holds for all $a,b \in D$.
\end{lemma}

\medskip

We now state the main result of this section but postpone its proof after two technical lemmas.

\begin{theor}\label{theor:diskew_sol}
    Let $(D,\vd,\dv,\circ)$ be a di-skew brace and define a map $r:D \times D \to D \times D$ by setting
    \begin{align*}
        r(a,b) = (\lambda_a(b),(\lambda_a(b))^- \circ (a \dv \lambda_a(b))) \,,
    \end{align*}
    for all $a,b \in D$.
    Then $(D,r)$ is a bijective non-degenerate solution.
    In particular, the derived shelf associated with $r$ is $\Conj(D,\vd,\dv)$.
\end{theor}

\medskip

\begin{lemma}\label{diskew_twist}
    Let $(D,\vd,\dv,\circ)$ be a di-skew brace and let $(D,\triangleright) = \Conj(D,\vd,\dv)$ be its conjugation rack.
    Then, the map $\lambda:D \to \Aut(D,\vd,\dv)$ is a twist for $(D,\triangleright)$.
    \begin{proof}
        Notice first that $\lambda_a \in \Aut(D,\triangleright)$, for all $a \in D$, since, in particular, it is an automorphism of the $g$-digroup.
        Therefore, to obtain the assertion, it sufficient to show that for all $a,b,c \in D$ the following holds:
        \begin{align*}
            \lambda_a\lambda_b
            &=
            \lambda_{\lambda_a(b)}\lambda_{\lambda^{-1}_{\lambda_a(b)}(\lambda_a(b) \triangleright a)}
        \end{align*}
        To facilitate the computation of the right-hand side, let us first note that
        \begin{align*}
            \lambda^{-1}_{\lambda_a(b)}(\lambda_a(b) \triangleright a)
            &= \lambda^{-1}_{\lambda_a(b)}((\lambda_a(b))^{-1} \vd a \dv \lambda_a(b)) \\
            &= (\lambda_a(b))^- \circ \left( \lambda_a(b) \vd (\lambda_a(b))^{-1} \vd a \dv \lambda_a(b) \right) \\
            &= (\lambda_a(b))^- \circ \left( a \dv \lambda_a(b) \right) \,,
        \end{align*}
        holds for all $a,b \in D$.
        From \cref{digroup_lambda}, it follows that the right-hand side of the required identity is equal to
        \begin{align*}
            \lambda_{a\, \dv\, \lambda_a(b)}(c)
            &= 
            \lambda_{a\, \vd\, \lambda_a(b)}(c)\,,
            \quad \text{applying \cref{rem:lambda_index},} \\
            &=
            \lambda_{a \circ b} \\
            &=
            \lambda_a\lambda_b(c) \,,
        \end{align*}
        thus concluding the proof.
    \end{proof}
\end{lemma}

\begin{lemma}\label{lemma:diskew_mult}
    Let $(D,\vd,\dv,\circ)$ be a di-skew brace and fix $\xi \in E(D)$.
    Under the decomposition provided by $\xi$, the following hold for all $a \in D$:
    \begin{align}
        g_{m_a} &= g_a \label{lem:diskew_add_mul1}\,,\\
        m_{g_a} &= m_a \label{lem:diskew_add_mul2}\,.
    \end{align}
    \begin{proof}
        Fix $a \in D$.
        For the first identity, note that 
        \begin{align*}
            m_a = a \circ \xi = a \vd a^{-1} \vd a \circ \xi = a \vd \lambda_a(\xi)
        \end{align*}
        which implies that $g_{m_a} = g_a$, since $\lambda_a(\xi) \in E(D)$.
        For the second identity, it is enough to note that
        \begin{align*}
            m_a = a \circ \xi = (g_a \vd e_a) \circ \xi
            \overeq{\eqref{diskew3}} (g_a \dv e_a)\circ\xi
            = g_a\circ\xi
            = m_{g_a} \,,
        \end{align*}
        obtaining the assertion.
    \end{proof}
\end{lemma}

\medskip

\begin{proof}[Proof of \cref{theor:diskew_sol}]
    First, let us observe that $r$ is a bijective and left non-degenerate solution, thanks to \cref{thm:twist_solution} and \cref{diskew_twist}.
    Therefore, in order to prove the right non-degeneracy of $r$, it sufficient to verify that the square map $\mathfrak{q}:D \to D$ of $(D,r)$ is bijective and to apply \cref{lem:rho_nondegenerate_sol}.
    With this aim, let us {first recall that under the notation of \cref{convenzione} if $a,b \in D$ then $a^- \circ b = m_a^- \circ b$ coincides with the unique element $x \in D$ such that $a \circ x = b$.
    If $a \in D$, it follows that}
    \begin{align*}
        \mathfrak{q}(a) 
        &= \lambda^{-1}_a(a) 
        = a^{-} \circ (a \vdash a) 
        = {m_a^{-} \circ (a \vdash a)}\\
        &\overset{\eqref{diskew1}}{=} m_a^{-}\circ a \vdash (m_a^{-})^{-1} \vdash m_a^{-} \circ a 
        = (m_a^{-})^{-1} \vdash u_a \,,
    \end{align*}
    where we have used the fact that $m_a^-\circ a = u_a\in E(D)$ {is a bar-unit}.
    Define the map $\mathfrak{p}:D \to D$ by setting $\mathfrak{p}(a) = ({g^{-1}_a})^{-} \circ e_a$, for all $a \in D$.
    Fix $a \in D$ and notice that
    \begin{align*}
        \mathfrak{q}(\mathfrak{p}(a)) 
        =\mathfrak{q}(m_{g^{-1}_a}^{-} \circ e_a)
        = (m_{g^{-1}_a})^{-1} \vd e_a 
        = g^{-1}_{m_{g^{-1}_a}} \vd e_a 
        \overeq{\eqref{lem:diskew_add_mul1}} g^{-1}_{g^{-1}_a}\vd e_a 
        = g_a \vd e_a
        = a \,,
    \end{align*}
    proving that $\mathfrak{q}\mathfrak{p} = \id_D$.
    Moreover, we also have that
    \begin{align*}
        \mathfrak{p}(\mathfrak{q}(a)) 
        =\mathfrak{p}(g_{m^{-}_a}^{-1} \vd u_a)
        = g^{-}_{m^{-}_a} \circ u_a
        = m^{-}_{g_{m^{-}_a}} \circ u_a
        \overeq{\eqref{lem:diskew_add_mul2}} m^{-}_{m^{-}_a} \circ u_a
        = m_a \circ u_a
        = a \,,
    \end{align*}
    proving that $\mathfrak{pq} = \id_D$.
    Therefore $\mathfrak{q}$ is invertible and thus bijective, with $\mathfrak{p}$ its inverse.
\end{proof}

\begin{rem}
    The left derived shelf associated to a skew brace solution is the conjugation quandle of its additive group (see, for instance, \cite{DoRySte24}).
    As a consequence, we deduce that if $(D,\vd,\dv,\circ)$ is a di-skew brace then its associated solution $r$ is not necessarily equivalent to a skew brace solution.
    Indeed, based on the fact that the left derived shelf of $r$ is $\Conj(D,\vd,\dv)$ {and thanks to \cref{rem:shelf_hom}}, to prove this assertion, it is enough to apply \cref{lemma:conj_rack_decomp} and consider a $g$-digroup with non-trivial conjugation action on the set of idempotents.
\end{rem}

{\begin{ex}
    The simplest di-skew brace solutions that do not come from a skew brace are those of left derived type $r(a,b) = (b, b^{-1} \vdash a \dashv b)$ defined on a g-digroup $(D,\vdash,\dashv)$ such that there exist $a \in D$ and $e \in E(D,\vdash,\dashv)$ satisfying the condition $a^{-1} \vdash e \dashv a \neq e$.

    An explicit example can be constructed as in \cref{ex:digroup}.
    Let $G = \langle g \rangle$ be the cyclic group of order 4 and consider its natural right action $\psi:G \to E$ on the set $E = \{1,2\}$ with $\psi_g = (12)$.
    Then $(G \times E,\vdash,\dashv)$ where the operations are defined by
    \begin{align*}
        (a,e) \vdash (b,f) = (ab,f) \,,\quad (a,e) \dashv (b,f) = (ab,\psi_b(e))
    \end{align*}
    for all $a,b \in G$ and $e,f \in E$, is a $g$-digroup belonging to the above class.
    \end{ex}}

\smallskip

\section{Compatible averaging operators on groups and examples}\label{Sec3}

The aim of this section is to show how certain maps defined on groups that we name \emph{compatible averaging operators} naturally lead to a rich class of examples of $g$-digroups. In particular, these maps include idempotent group endomorphisms and also yield an automatic identification of the set of bar units of the $g$-digroups they provide. Lastly, we show how to obtain di-skew braces having as additive structure a $g$-digroup that comes from compatible averaging operators.
\smallskip

\subsection{Averaging operators on groups}
Let us begin by introducing some definitions and properties with some due remarks.

\begin{defin}
    Let $D$ be a group.
    A \emph{left} (resp. \emph{right}) \emph{averaging operator on $D$} is a map $f:D \to D$ such that the following hold:
    \begin{align}\label{eq:la-ra-op}
        \forall\, x,y \in D\qquad 
        f\left(x\right) f\left(y\right)
        = f\left(f\left(x\right)\, y\right)
        \qquad \text{(resp.} \ f\left(x\right) f\left(y\right)
        = f\left(x\, f\left(y\right)\right)\text{)}.
    \end{align}
    A map $f$ is called an \emph{averaging operator} on $D$ if it is both a left and right averaging operator. Note that this last notion was introduced in \cite{ZhGa24x}.
\end{defin}

\begin{ex}
    Let $D$ be a group, let $c \in D$ and define a map $f:D \to D$ by setting $f(x) = xc$, for all $x \in D$.
    If $c \notin Z(D)$, then $f$ is a left averaging operator which is not a right averaging operator.
    Analogously, the map $g:D \to D$ defined by setting $g(x) = cx$, for all $x \in D$, is a right averaging operator which is not a left averaging operator.
\end{ex}

Let us observe that every idempotent homomorphism on a group $D$ is an averaging operator on $D$. Moreover, note that if $f$ is either a left or right averaging operator on a group $D$ then $f\left(1_G\right) = 1_G$ (with $1_G$ the identity of $D$) if and only if $f$ is an idempotent map. On the other hand, there exist averaging operators that are not endomorphism.

\begin{ex}
    Let $D$ be a group {that admits a subgroup $N$ of index $2$ and assume that there exists $c \in C_D(N)$ such that $c \notin N$.}
    Define $f:D \to D$ by setting $f(x) = x$ if $x \in N$ and $f(x) = xc$ if $x\notin N$. 
    Then, $f$ is an averaging operator on $D$.
\end{ex}

\medskip

In the following result, we provide a method to construct $g$-digroups on a fixed group which employs a slight generalization of the notion of averaging operators on groups and is inspired by results in \cite[Proposition~5.6]{ChMa25x} and \cite[Proposition~3.8]{ZhGa24x}.
Moreover, to ease the notation, if $f$ is a left (resp. right) averaging operator on a group $D$ then we set $\ker(f) = \{x \in D \mid f(x) = 1\}$.
It is worth mentioning that in general $\ker(f)$ is not necessarily a subgroup of $D$.

\begin{defin}
    Let $D$ be a group and $f,g:D \to D$ maps. The pair $(f,g)$ is called an \emph{averaging pair} if the following identities 
    \begin{align}
        f(a)& f(b)
        = f(f(a)\,b)
        = f(a\, g(b))  \label{ex_KU_1}\\
        g(a)& g(b)
        = g(a\, g(b))
        = g(f(a)\, b)\label{ex_KU_2}
\end{align}
    hold, for all $a,b \in D$. Moreover, we say that $(f,g)$ is a \emph{compatible averaging pair} if $\ker(f)\cap\ker(g)\neq\emptyset$.
\end{defin}

\begin{prop}\label{ex_KU}
Let $D$ be a group and $(f,g):D\to D$ an averaging pair on $D$. Then, the structure $\left(D,\,\vd,\,\dv\right)$ is a disemigroup where we define
\begin{align*}
    a \vd b:= f(a)\, b
    \qquad\text{and}\qquad
    a \dv b:= a\, g(b),
\end{align*}
for all $a,b \in D$. Moreover, if $\left(f,g\right)$ is compatible, then $\left(D,\,\vd,\,\dv\right)$ is a $g$-digroup whose set of bar-units $E(D)$ coincide with $\ker(f) \cap \ker(g)$ and whose unilateral inverses are given by $I_e(a):= e\,g(a)^{-1}$ and $J_e(a):= f(a)^{-1}\,e$, for all $e \in E(D)$ and $a \in D$.
\begin{proof}
    The proof consists of a routine computation.
\end{proof}
\end{prop}

Notice that if $f:D \to D$ is an averaging operator on a group $D$, then $(f,f)$ is trivially an averaging pair which is compatible if and only if $\ker(f) \neq \emptyset$.
Therefore, every such map may serve to construct examples of $g$-digroups.

\begin{ex}\label{ex-Kl}
    Let $f,g:D \to D$ be idempotent endomorphisms of a group $D$ such that\ $fg = f$\ and\ $gf = g$.
    Then, $(f,g)$ is a compatible averaging pair on $D$.
    Notice that the previous two identities are equivalent to the fact that $\ker(f) = \ker(g)$, as it is clear by employing the natural factorization induced by $f$ and $g$ on $D$.
    As an example of such a pair, we can consider $D =\{1,\, a,\, b,\, ab\}$ the Klein group and the idempotent endomorphisms $f,g:D\to D$ defined by setting $f\left(a\right) = 1$, $f\left(b\right) = b$, and $g\left(a\right) = 1$, $g\left(b\right) = ab$, respectively. 
\end{ex}

\begin{rem}\label{rem-isom-g-dgroups}
    Let $(f,g)$ be a compatible averaging pair on a group $D_1$ and consider its associated $g$-digroup $(D_1,\vd,\dv)$.
    Moreover, consider an averaging operator $h$ on a group $D_2$ such that $\ker(h) \neq \emptyset$ and denote by $(D_2,\vd,\dv)$ its associated $g$-digroup.
    If $\varphi:(D_1,\vd,\dv) \to (D_2,\vd,\dv)$ is a $g$-digroup isomorphism, then easy calculations show that the identities
    \begin{align*}
        \varphi f(a) = h(\varphi(a)) \varphi(1)
        \qquad\text{and}\qquad
        \varphi g(a) = \varphi(1) h(\varphi(a)) \,,
    \end{align*}
    hold for all $a \in D_1$.
    In particular, we deduce that $f=g$ if and only if $\varphi\left(1\right)\in C_{D_2}\left(\text{Im}(h)\right)$.
    Following this reasoning, the example defined on the Klein group in \cref{ex-Kl} provides a concrete instance of a compatible averaging pair that cannot be obtained by a single averaging operator. 
\end{rem}

Let us conclude with a further example of a class of $g$-digroups that do not necessarily rise from averaging operators.

\begin{ex}\label{ex:fg=gf}
    Let $D$ be a group, $f,g$ idempotent endomorphisms of $D$ such that $fg = gf$ and $\text{Im} (f)$ is an abelian group. Then $\left(D,\, \vd,\, \dv\right)$ is a $g$-digroup with
    \begin{align*}
        a\vd b = f(a)fg(a)^{-1}b
        \qquad\text{and}\qquad
        a\dv b = af(b)fg(b)^{-1},
    \end{align*}
    for all $a,b\in D$ such that $E(D)=\{a\in D\mid f(a) = fg(a)\}$.
\end{ex}

\medskip

\subsection{Constructions of di-skew braces through averaging operators}

In this section, we provide concrete instances of di-skew brace structures by employing averaging operators on groups as a tool.
\medskip

Before proceeding, let us first observe that if $D$ is a group, $f$ a left averaging operator and $e \in \ker(f)$ then a simple computation shows that 
    \begin{align}
        f(a)f(f(a)^{-1}\,e) = 1\qquad\text{and}\qquad 
        g(e\, g(a)^{-1}) g(a) = 1 \label{eq:averag_op_rem}
    \end{align}
    hold, for all $a,b \in D$.

\smallskip

\begin{prop}\label{prop:avg_diskewbrace}
    Let $D$ be a $g$-digroup obtained by a compatible averaging pair $(f,g)$  on a group $D$, and consider a left averaging operator $h$ on $D$. If we define a further binary operation $\circ$ on $D$ by setting
    \begin{align*}
        \forall\, a,b\in D \quad
        a \circ b := h(a)\,b
    \end{align*}
    then the structure $(D,\vd,\dv,\circ)$ is a di-skew brace if and only if the following identities
    \begin{align}
    h(a)f(b) &= f(h(a)b) f(a)^{-1} h(a) \label{avg_diskew1}\tag{AD1}\\
    g(a)g(b) &= g(h(a)b) \label{avg_diskew2}\tag{AD2}\\
    h(f(a)b) &= h(ag(b)) \label{avg_diskew3}\tag{AD3}
\end{align}
    hold, for all $a,b\in D$.
    \begin{proof}  
        The proof relies on verifying that the identities \eqref{avg_diskew1}, \eqref{avg_diskew2} and \eqref{avg_diskew3} can be obtained respectively by rewriting \eqref{diskew1}, \eqref{diskew2} and \eqref{diskew3} in terms of $f$, $g$ and $h$, {using the equalities \eqref{eq:averag_op_rem}}.    
    \end{proof}
\end{prop}

\smallskip

Note that if $(D,\vd,\dv)$ is a $g$-digroup obtained through a compatible averaging pair $(f,g)$, then the trivial di-skew brace on $D$ can be constructed with \cref{prop:avg_diskewbrace} by setting $h = f$.
On the other hand, the almost trivial di-skew brace on $D$ is recovered similarly as shown below, where we provide a dual version of \cref{prop:avg_diskewbrace}.

\begin{prop}\label{prop:avg_diskewbrace2}
    Let $D$ be a $g$-digroup obtained by a compatible averaging pair $(f,g)$ on a group $D$ and consider a right averaging operator $h$ on $D$. If we define a further binary operation $\circ$ on $D$ by setting
    \begin{align*}
        \forall\, a,b\in D \quad
        a \circ b := b\,h(a)
    \end{align*}
    then the structure $(D,\vd,\dv,\circ)$ is a di-skew brace if and only if the following identities
    \begin{align}
        f(b)f(a) &= f(b\,h(a))\,,\label{avg_diskew1-2}\tag{AD1$'$}\\
        g(b)h(a) &= h(a)g(a)^{-1}g(b\,h(a))\,,\label{avg_diskew2-2}\tag{AD2$'$}\\
        h(f(a)\,b) &= h(a\,g(b))\,,\label{avg_diskew3-2}\tag{AD3$'$}
    \end{align}
    hold, for all $a,b \in D$.
\end{prop}

\smallskip 

\begin{rem}\label{rem:avg_diskew_sol}
    It is worth computing explicitly the solutions determined thanks to \cref{prop:avg_diskewbrace}.
    With this aim, let $(f,g)$ be a compatible averaging pair on a group $D$ and $h$ a left averaging operator as in \cref{prop:avg_diskewbrace}.
    Let $r$ be the solution associated with the di-skew brace $(D,\vd,\dv,\circ)$ provided by $(f,g)$ thanks to \cref{theor:diskew_sol}.
    Then, an easy application of \eqref{eq:averag_op_rem} yields
    \begin{align*}
        r(x,y) =
        \left(f(x)^{-1} h(x)\, y, \ 
        h(f(x)^{-1}h(x)\,y)^{-1} x\, g(f(x)^{-1}h(x)y)
        \right)
    \end{align*}
    for all $x,y \in D$. 
    Similarly, if $h:D \to D$ is a right averaging operator on $D$, then in the context and notation of \cref{prop:avg_diskewbrace2} it determines a di-skew brace structure whose associated solution is given by
    \begin{align*}
        r(x,y) =
        \left(f(x)^{-1}yh(x) \,,\,
        x\,g(f(x)^{-1}yh(x))\, h(f(x)^{-1}yh(x))^{-1}
        \right),
    \end{align*}
    for all $x,y \in D$.
\end{rem}

\smallskip

In the notation of \cref{prop:avg_diskewbrace} we provide the following class of examples.

\begin{ex}\label{ex:avg_diskewbrace}
    Let $f,g$ be idempotent endomorphisms on a group $D$ such that $fg = f$ and $gf = g$ and let $(D,\vd,\dv)$ be the associated $g$-digroup.
    Consider a further idempotent endomorphism $h$ on $D$ such that $\ker(f) = \ker(h) = \ker(g)$.
        Then, following \cref{prop:avg_diskewbrace}, we deduce that $(D,\vd,\dv,\circ)$ is a di-skew brace whose associated solution is given by
        \begin{align*}
            r(x,y) = \left(f(x)^{-1}h(x)\,y,\, h(y)^{-1} x\,g(y)\right),
        \end{align*}
        for all $x,y\in D$.
\end{ex}

\begin{rem}
    As a final remark, let us point out that the solutions obtained through \cref{theor:diskew_sol} do not necessarily satisfy an identity analogous to the Lu--Yan--Zhu condition \cite{LuYZ00}.
    Indeed, consider $D =\{1,\, a,\, b,\, ab\}$ the Klein group and let $f,h:D\to D$ be defined by setting $f(a) = 1$, $f(b) = b$, and $h(a) = 1$, $h(b) = ab$, respectively. 
    Then, $f,h$ are idempotent endomorphisms on $D$ satisfying $\ker(f) = \ker(h)$.
    {If we consider the di-skew brace induced by the compatible pair $(f,f)$ together with $h$ as in \cref{ex:avg_diskewbrace}, then the following holds:
    \begin{align*}
        \lambda_a(b) \circ \rho_b(a) &= h\left(\lambda_a(b)\right)\rho_b(a) 
        = h\left(f(a)^{-1}h(a) b\right) h(b)^{-1} a f(b) \\
        &= h(b) h(b)^{-1} a f(b)
        = a f(b) 
        = a b \,,
    \end{align*}
    while $a \circ b = h(a)b = b$, thus proving the assertion.}
\end{rem}

\smallskip

\section{The hemi-semidirect product of solutions}\label{Sec4}

The aim of this section is to introduce the construction of \emph{hemi-semidirect product of twists} and prove that di-skew brace solutions can be encoded in this way.
In doing this, we also show that our notion is an instance of \emph{dynamical extension} of solutions, which is a general method to construct new solutions based on an analogous technique for self-distributive structures as in \cite{AndGr03}.
We point out that our treatment is equivalent to  \cite{CasCatSte22} in the setting of $q$-cycle set but is suitably translated in the language of twists, thus highlighting the role of the underlying shelf structure.
Lastly, as an application, we delve deeper in the properties of di-skew brace solutions.

\medskip

\begin{defin}\label{def:dynamical-pair}
    Let $(G,\tr)$ be a shelf equipped with an admissible twist $\lambda:G \to \Aut(G,\tr)$ and let $E$ be a set.
    A pair $\left(\alpha,\beta\right)$ is said to be a \emph{dynamical pair on $G$ with coefficients in~$E$} if $\alpha,\beta: G\times G \to \Map(E \times E, E)$ are maps such that $\beta_{a,b}(s,{-}):E \to E$ is bijective, for all $a,b \in G$, $s \in E$, and the identities
        \begin{align*}
            \alpha_{a,b \tr c}
            \left( u,\alpha_{b,c}(s,t) \right) 
            &=
            \alpha_{a \tr b,a \tr c}
            \left( \alpha_{a,b}(u,s),\alpha_{a,c}(u,t) \right) 
            \,,\\[0.2cm]
            \alpha_{a \cdot b, a \cdot c}
            \left( \beta_{a,b}(u,s), \beta_{a,c}(u,t) \right)
            &=
            \beta_{a,b \tr c}
            \left( u, \alpha_{b,c}(s,t) \right)
            \,,\\[0.2cm]
            \beta_{a \cdot b, a \cdot c}
            \left( \beta_{a,b}(u,s), \beta_{a,c}(u,t) \right)
            &=
            \beta_{b \cdot (b \tr a), b \cdot c}
            \left( \beta_{b,b \tr a}(s,\alpha_{b,a}(s,u)), \beta_{b,c}(s,t) \right) \,,
        \end{align*}
        hold, for all $a,b,c \in G$, $u,s,t \in E$, where we set $x \cdot y := \lambda^{-1}_x(y)$, for all $x,y \in G$.
\end{defin}

Note that if~$(G,\tr)$ is a shelf then a map~$\alpha:G \times G \to \Map(E \times E,E)$ satisfying the first identity of \cref{def:dynamical-pair} is called a \emph{dynamical $2$-cocycle} on $G$ with coefficients in $E$.
In this context, thanks to \cite[Lemma~2.1]{AndGr03}, the binary operation $\tr_\alpha$ on $G \times E$ defined by setting 
\begin{align}\label{eq:dynamical_shelf}
    (a,s) \tr_\alpha (b,t) &= (a \tr b, \alpha_{a,b}(s,t)) \,,
\end{align}
for all $a,b \in G, s,t \in E$, endows $G \times E$ with a shelf structure called a \emph{dynamical (shelf) extension} of $G$ (via $\alpha$ with coefficients in $E$).

\begin{rem}
    The relationship between admissible twists and $q$-cycle sets has already been described in \cite[Remark 2.19]{DoRySte24}.
    At this point we limit ourselves to illustrate the connection between \cref{def:dynamical-pair} and \cite[Definition 2]{CasCatSte22}, where $q$-cycle set dynamical pairs (or $q$-dynamical pairs for short) are introduced.
    
    Let $(G,\tr)$ be a shelf equipped with an admissible twist~$\lambda:G \to \Aut(G)$ and let~$(\alpha,\beta)$ be a dynamical pair with coefficients in a set $E$.
    The map ${\gamma':G \times G \to \Map(E \times E,E)}$ defined by setting
    \begin{align}\label{skew-cycle to q-cycle}
        \gamma'_{x,y}(s,t) = \beta_{x,x \tr y}(s, \alpha_{x,y}(s,t)) \,,
    \end{align}
    for all $a,b \in G$, $s,t \in E$, determines a pair $(\beta,\gamma')$ which is a $q$-dynamical pair.
    Conversely, let $(G,\cdot,:)$ be a $q$-cycle set and consider a $q$-dynamical pair $(\gamma,\gamma')$ with coefficients in $E$.
    If we define a map $\alpha: G \times G \to \Map(E \times E,E)$ by setting
    \begin{align}\label{q-cycle to skew-cycle}
        {\alpha_{x,y}(s,t) = \gamma_{x,y}^{-1}\left(s,\gamma'_{x,y}(s,t)\right)} \,,
    \end{align}
    for all $a,b \in G, s,t \in E$, then the pair $(\alpha,\gamma)$ is a dynamical pair.
\end{rem}

\medskip

Before providing \cref{theor:twist_extension}, one of the main results of the section, we prove a small technical lemma.

\begin{lemma}\label{lemma:YB1_shelf}
    Let $(G,\tr)$ be a shelf and let $\lambda:G \to \Aut(G,\tr)$ be a map.
    Then, $\lambda$ is a twist (i.e. satisfies \eqref{eq:twist}) if and only if the the following holds
    \begin{align}\label{eq:equi_first}
        \forall\, a,b \in G \qquad
        \lambda^{-1}_{\lambda^{-1}_a\left(b\right)}\lambda^{-1}_a
        = 
        \lambda^{-1}_{\lambda^{-1}_b\left(b\,\tr\, a\right)}
        \lambda^{-1}_b
    \end{align}
    \begin{proof}
        Note that \eqref{eq:twist} is equivalent to requiring that $\lambda_a\lambda_{\lambda^{-1}_a(b)} = \lambda_b\lambda_{\lambda^{-1}_b(b \tr a)}$ holds, for all $a,b \in G$.
        Therefore, the assertion easily follows after inverting both sides of the above identity.
    \end{proof}
\end{lemma}

\begin{theor}\label{theor:twist_extension}
    Let $(G,\tr)$ be a shelf equipped with an admissible twist {$\lambda:G \to \Aut(G,\triangleright)$} and consider a dynamical pair $(\alpha,\beta)$ of $G$ with coefficients in a set $E$.
    Let $(G \times E,\tr_\alpha)$ be the dynamical extension of $G$ via $\alpha$ and define a map {$\Lambda_{(a,s)}:G\times E\to G\times E$} by setting 
    \begin{align}\label{eq:twist_extension_lambda}
        {\Lambda}_{(a,s)}(b,t) &= \left(\lambda_a(b), \beta^{-1}_{a,\lambda_a(b)}(s,t)\right) \,,
    \end{align}
    for all $a,b \in G$ and $s,t \in E$.
    Then, the map {$\Lambda:G \times E \to \Map(G \times E, G \times E)$} is an admissible twist for $(G \times E,\tr_\alpha)$.
    \begin{proof}
    In the setting of the statement, define a map ${\Lambda_{(a,u)}}:G \times E \to G \times E$, for a fixed pair $(a,u) \in G \times E$, by equation \eqref{eq:twist_extension_lambda}.
    {If we consider the map $\Gamma_{(a,u)}:G \times E \to G \times E$ defined by setting $\Gamma_{(a,u)}(b,v) = \left(\lambda^{-1}_a(b),\beta_{a,b}(u,v)\right)$, for all $b \in G$ and $v \in E$, then the following holds:
    \begin{align*}
        \Lambda_{(a,u)}\Gamma_{(a,u)}(b,v) 
        =
        \left(\lambda_a\lambda_a^{-1}(b),\beta^{-1}_{a,\lambda_a\lambda_a^{-1}(b)}\left(u,\beta_{a,b}(u,v)\right)\right) 
        =
        \left(b,\beta^{-1}_{a,b}\left(u,\beta_{a,b}(u,v)\right)\right) \,,
    \end{align*}
    for all $b \in G$ and $v \in E$, where the last term is equal to $(b,v)$ since the maps $\beta_{a,b}(u,\boldsymbol{\cdot})$ and $\beta^{-1}_{a,b}(u,\boldsymbol{\cdot})$ are inverses of each other by construction.
    With similar calculations it is possible to verify that $\Gamma_{(a,u)}\Lambda_{(a,u)} = \id_{G \times E}$, thus proving that $\Lambda_{(a,u)}$ is invertible with inverse $\Gamma_{(a,u)}$.}
    Moreover, the fact that  ${\Lambda}_{(a,u)} $ belongs to {$\Aut(G\times E,\tr_\alpha)$} is clearly equivalent {to the requirement that} 
    \[
        {\Lambda^{-1}_{(a,u)}}\left((b,v) \tr_\alpha (c,z)\right) 
        = {\Lambda^{-1}_{(a,u)}}(b,v) \tr_\alpha {\Lambda^{-1}_{(a,u)}}(c,z)
    \]
    holds, for all $b,c \in G, v,z \in E$, which in turn is easily seen to be equivalent to the second identity of {\cref{def:dynamical-pair}}.
    To conclude the proof, thanks to \cref{lemma:YB1_shelf} it remains to show that ${\Lambda}$ satisfies \eqref{eq:equi_first}.
    With this aim, fix $a,b,c \in G, u,v,z \in E$ and notice that the left-hand-side in \eqref{eq:equi_first} reads
    \begin{align}\label{eq:I}\tag{I}
        {\Lambda}^{-1}_{\lambda^{-1}_{(a,u)}(b,v)}{\Lambda}^{-1}_{(a,u)}(c,z)
        =
        {\Lambda}^{-1}_{\left(\lambda^{-1}_a(b),\,\beta_{a,b}(u,v)\right)}
        \left(\lambda^{-1}_a(c),\beta_{a,c}(u,z)\right) \,.
    \end{align}
    The first component of \eqref{eq:I} is 
    $\lambda^{-1}_{\lambda^{-1}_a(b)}\lambda^{-1}_a(c)$
    and the second one is
    \begin{align*}
        \beta_{\lambda^{-1}_a(b),\lambda^{-1}_a(c)}
        \left(\beta_{a,b}(u,v), \beta_{a,c}(u,z)\right).
    \end{align*}
    Now, if we note that 
    \begin{align*}
           {\Lambda}^{-1}_{(b,v)}\left((b,v) \tr_\alpha (a,u)\right) 
            =\left(\lambda^{-1}_b(b \tr a), \beta_{b,b \tr a}\left(v,\alpha_{b,a}(v,u)\right)\right)
            =: U \,,
    \end{align*}
    then the right-hand-side in \eqref{eq:equi_first} reads
    \begin{align}\label{eq:II}\tag{II}
        {\Lambda}^{-1}_{\lambda^{-1}_{(b,v)}\left((b,v) \tr_\alpha (a,u)\right)}{\Lambda}^{-1}_{(b,v)}(c,z)
        =
        {\Lambda}^{-1}_{U}
        \left(\lambda^{-1}_b(c),\beta_{b,c}(v,z)\right).
    \end{align}
    The first component of \eqref{eq:II} is equal to
    $\lambda^{-1}_{\lambda^{-1}_b\left(b\tr a\right)}
    \lambda^{-1}_b(c)$,
    while the second component of $\eqref{eq:II}$ is
    \begin{align*}
        \beta_{\lambda^{-1}_b(b \tr a),\lambda^{-1}_b(c)}
        \left(\beta_{b,b \tr a}\left(v,\alpha_{b,a}(v,u)\right),\beta_{b,c}(v,z)\right) \,.
    \end{align*}
    Therefore, the assertion is proved by comparing the components of \eqref{eq:I} and \eqref{eq:II}.
\end{proof}
\end{theor}

\bigskip

We now specialize the results obtained above to introduce the main topic of this section.

\begin{defin}
    Let $(G, \tr)$ be a shelf and let $E$ be a set.
    An \emph{action of $(G,\tr)$ on $E$} is a map $\psi:G \to \Sym(E) 
    \,, a \mapsto \psi_a$ such that the following holds
    \begin{align}\label{eq:shelf_action}
        \forall\, a,b \in G\quad \psi_{a \triangleright b}\psi_a = \psi_a \psi_b \,,
    \end{align}
\end{defin}

In the notation of the above definition, the binary operation $\tr_\psi$ defined on $G \times E$ by setting 
\begin{align*}
    (a,u) \tr_\psi (b,v) := (a \tr b,\, \psi_a(v)) \,,
\end{align*}
for all $a,b \in G, u,v \in E$, determines a shelf structure $(G \times E, \tr_\psi)$ which is known as the \emph{hemi-semidirect product of $(G,\tr)$ and $E$ via $\psi$}. For more details, we refer the reader to~\cite{CrWag14} and~\cite{Kin07}.
The following definition is justified by \cref{prop:twist_hemi_product}.

\medskip
{
Let us briefly recall that a twist $\lambda:G \to \Aut(G,\tr)$ determines a left non-degenerate solution on $G$ defined by $r(a,b) = (\lambda_a(b),\rho_b(a))$, where
\begin{align}\label{eq:rho_map}
    \rho_b(a) = \lambda_{\lambda_a(b)}^{-1}\left(\lambda_a(b) \tr a \right) \,,
\end{align}
for all $a,b \in G$ (cf. \cref{thm:twist_solution}).}

\begin{defin}
    Let $(G,\tr)$ be a shelf equipped with a twist $\lambda:G \to \Aut(G,\tr)$ and let $\psi$ be an action of $(G,\tr)$ on a set $E$. 
    Consider a map $\sigma:G \to \Sym(E)$ satisfying
    \begin{align} 
            \sigma_a\sigma_b &= \sigma_{\lambda_a(b)}\sigma_{\rho_b(a)} \label{eq:hemi_sol1} \, \\
            \sigma_a\psi_b &= \psi_{\lambda_a(b)}\sigma_a \label{eq:hemi_sol2} \,
    \end{align}
    for all $a,b \in G$.
    Then, the pair $(\psi,\sigma)$ is called a \emph{hemi-semidirect pair of $\lambda$ with coefficients in $E$}.
\end{defin}

\begin{prop}\label{prop:twist_hemi_product}
    Let $(G,\tr)$ be a shelf equipped with a twist $\lambda:G \to \Aut(G)$, let $E$ be a set and consider two maps $\psi:G \to \Sym(E)$ and $\sigma:G \to \Sym(E)$.
    If for all $a \in G, u \in E$ we define ${\Lambda}_{(a,u)}:G \times E \to G \times E$ by setting
    \begin{align*}
        {\Lambda}_{(a,u)}(b,v) := (\lambda_a(b),\sigma_a(v)) \,,
    \end{align*}
    for all $b \in G, v \in E$, then the map ${\Lambda}:G \times E \to \Aut(G \times E, \tr_\psi),\, (a,u) \mapsto \lambda_{(a,u)}$ is a twist if and only if $(\psi,\sigma)$ is a hemi-semidirect pair on $G$ with coefficients in $E$.
    \begin{proof}
        For all $a,b \in G$ define two maps $\alpha_{a,b}$, $\beta_{a,b}:G \times G \to \Map(E \times E, E)$ by setting $\alpha_{a,b}(u,v) := \psi_a(v)$ and $\beta_{a,b}(u,v) := \sigma_a(v)$, for all $u,v \in E$.
        A straightforward computation shows that $(\alpha,\beta)$ is a dynamical pair if and only if $(\psi,\sigma)$ is a hemi-semidirect pair, thus proving the assertion thanks to \cref{theor:twist_extension}.
    \end{proof}
\end{prop}

It is worth to remark that \cref{prop:twist_hemi_product} is thus configured as an extension of the hemi-semidirect product from the setting of self-distributive structures to that of set-theoretic solutions.
Under this light and in the notation of \cref{prop:twist_hemi_product}, if $\lambda_G$ is a twist on a shelf $(G,\triangleright)$ and if $(\psi,\sigma)$ is a hemi-semidirect pair on $G$ then, denoted by $r_G$ the solution obtained by the twist $\lambda_G$, then the twist ${\Lambda}:G \times E \to \Aut(G \times E, \tr_{\psi})$ is called the \emph{hemi-semidirect product of $r_G$ and $E$ via $(\psi,\sigma)$}.
By a slight abuse of notation, we also let the left non-degenerate solution $r_D$ induced on $D := G \times E$ bear the same name and denote it by 
\begin{align*}
    r_D = r_G\, \text{\scalebox{0.9}{${\dv\vd}$}$_{\psi,\sigma}$}\, r_E \,,
\end{align*}
where we interpret $E$ as equipped with the trivial solution $r_E$.

\medskip

Hereinafter, when clear from the context we denote with $L$ the left multiplication operator of a shelf without explicit mention to the structure it belongs to.

\begin{cor}\label{corol:hemi_product_sol_order}
    Let $(G,r_G)$ be a bijective left non-degenerate solution,
    consider a hemi-semidirect product of $r_G$ and $E$ via $(\psi,\sigma)$ on $G$ with coefficients in a set $E$, with $|E| > 1$, and set 
    \begin{align*}
        m_{\tr_r} 
        &:= \textnormal{inf}\left\{h \in \mathbb{N}\mid \forall\, a,b \in G \quad (L_bL_a)^hL_b^{-h}(b) = b\right\}\,,\\[0.2cm]
        m_\psi 
        &:= \textnormal{inf}\left\{h \in \mathbb{N}\mid \forall\, a,b \in G,\, \forall\, v \in E \quad (\psi_b\psi_a)^h\psi_b^{-h}(v) = v\right\} \,.
    \end{align*}
    Then, $r_G \, \text{\scalebox{0.9}{${\dv\vd}$}$_{\psi,\sigma}$}\, r_E$ has finite order if and only if $m_{\tr_r}$ and $m_\psi$ are both finite.
    In particular, when this happens we have that $o(r_G \, \text{\scalebox{0.9}{${\dv\vd}$}$_{\psi,\sigma}$}\, r_E) = 2\,\textnormal{lcm}(m_{\tr_r}\,,m_\psi)$.
    \begin{proof}
        With the aim of applying \cref{theor:derived_sol_order}, let $a,b \in G$ and $u,v \in E$, set $A:= (a,u)$ and $B:= (b,v)$ and note that
        \begin{align*}
            (L_BL_A)^nL_B^{-n}(B) = 
            \left( (L_bL_a)^nL_b^{-n}(b), (\psi_b\psi_a)^n\psi_b^{-n}(v) \right)
        \end{align*}
        holds, for all $n\in \mathbb{N}$. 
        Now, if there exists $n\in \mathbb{N}$ such that $(L_BL_A)^nL_B^{-n}(B) = A$, for all $a,b \in G$ and $u,v \in E$, then in particular it follows that $(\psi_b\psi_a)^n\psi_b^{-n}$ is constant, a contradiction.
        Therefore, given the hypothesis and thanks to \cref{theor:derived_sol_order}, we deduce that $r_G \, \text{\scalebox{0.9}{${\dv\vd}$}$_{\psi,\sigma}$}\, r_E$ has finite order if and only if $\theta_{\tr_r}$ and $\theta_\psi$ are both finite and that the former is equal to
        \begin{align*}
            2\; \textnormal{min}\left\{n\in\mathbb{N}\mid\forall\,a,b \in G,\, \forall\,u,v \in E\quad \left( (L_bL_a)^nL_b^{-n}(b), (\psi_b\psi_a)^n\psi_b^{-n}(v) \right) = (b,v) \right\}\,,
        \end{align*}
        from which follows that $o(r_G \, \text{\scalebox{0.9}{${\dv\vd}$}$_{\psi,\sigma}$}\, r_E) = 2\;\textnormal{lcm}(m_{\tr_r}\,,m_\psi)$.
    \end{proof}
\end{cor}

\smallskip

\begin{ex}\,
    \begin{enumerate}
        \item[1)] If $(G,r_G)$ is a left non-degenerate bijective solution, then $(L,\,\lambda)$ is a hemi-semidirect pair on $G$ with coefficients in itself.
        \item[2)] Let $(G,r_G)$ be a finite bijective non-degenerate solution having even order and $E$ a non-empty set.
        Furthermore, fix $f,g \in \Sym(E)$ such that $fg = gf$ and define two maps $\psi,\sigma:D \to \Sym(E)$ by setting $\psi_a = f$ and $\sigma_a = g$, for all $a \in D$.
        Then, $(\psi,\sigma)$ is a hemi-semidirect pair on $G$ with coefficients in $E$.
        Moreover, the hemi-semidirect product of $r_G$ and $E$ via $(\psi,\sigma)$ satisfies $o(r_G \, \text{\scalebox{0.9}{${\dv\vd}$}$_{\psi,\sigma}$}\, r_E) = 2 \textnormal{lcm}(m_{\tr_r},o(f))$.
    \end{enumerate}
\end{ex}

\smallskip

We now show that the solution associated with a di-skew brace is indeed a hemi-semidirect product of its groupal and idempotent components. 
To prove this result, in the decomposition of a $g$-digroup - with respect to a fixed bar-unit $0 \in E$ - we give priority to the left groupal part $G$ and briefly refer to this fact by writing $D = G \,\text{\scalebox{0.9}{$\dv\vd$}}\, E$.

\begin{theor}\label{theor:diskew_twist_decomp}
    Let $(D,\vd,\dv,\circ)$ be a di-skew brace with $D = G \,\text{\scalebox{0.9}{$\dv\vd$}}\, E$ and let $r$ be its associated solution.
    Then, there exists a skew brace structure on $G$ such that
    \begin{align*}
        r = r_G\, \text{\scalebox{0.9}{${\dv\vd}$}$_{\psi,\sigma}$}\, r_E \,,
    \end{align*}
    where $(\psi,\sigma)$ is a hemi-semidirect pair on $G$ with coefficients in $E$.
    \begin{proof}
        Let $D = G \,\text{\scalebox{0.9}{$\dv\vd$}}\, E$ and define a binary operation ${\bullet}$ on $G$ by setting $a \bullet b := a \circ b \vd 0$, for all $a,b \in G$.
        It is straightforward to verify that $(G,\vd,\bullet)$ is a skew brace. Moreover, following \cref{rem:diskew_decomp}, if $\lambda:(D,\circ) \to \Aut(D,\vd,\dv)$ is the lambda map of $D$ then the twist associated with $(G,\vd,\bullet)$ is the map $\varphi:(G,\bullet) \to \Aut(G,\vd)$ defined by setting $\varphi_a(b) = \lambda_a(b) \vd 0$, for all $a,b \in G$.
        Now, let us define two maps $\psi,\sigma:G \to \Map(E,E)$ by setting
        \begin{align*}
            \psi_g(e) := g \vd e \dv g^{-1}
            \quad\textnormal{and}\quad
            \sigma_g(e) := \lambda_g(e) \,,
        \end{align*}
        for all $g \in G$ and $e \in E$ and let us prove that $(\psi,\sigma)$ is a hemi-semidirect pair of $\varphi$ with coefficients in $E$.
        With this aim, let $a,b \in G$, $u \in E$ and note that
        \begin{align*}
            \sigma_a\psi_b(u) 
            &=
            \lambda_a(b \vd u \dv b^{-1})
            =
            \lambda_a(b) \vd \lambda_a(u) \dv \lambda_a^{-1}(b)\\
            &=
            \varphi_a(b) \vd \sigma_a(u) \dv \varphi_a^{-1}(b)\\
            &=
            \psi_{\varphi_a(b)}\sigma_a(u) \,,
        \end{align*}
        thus showing that \eqref{eq:hemi_sol2} holds.
        {Now, if we let $\rho_b:G \to G$ be defined by \eqref{eq:rho_map} with respect to the twist $\varphi:G \to \Aut(G,\vd)$, for all $b \in G$, then the following holds
        \begin{align*}
            \sigma_{\varphi_a(b)}\sigma_{\rho_b(a)}(u) 
            =
            \sigma_{a \bullet b}(u)
            =
            \lambda_a\lambda_b(u)
            =
            \sigma_a\sigma_b(u) \,,
        \end{align*}
        for all $a,b \in G$ and $u \in E$,} thus obtaining \eqref{eq:hemi_sol1}.
        To finish the proof, it is sufficient to define $F:D \to G \times E$ by setting $F(a) = (g_a,e_a)$, for all $a \in D$, and verify that it is indeed an equivalence of solutions between $r_D$ and $r_G\, \text{\scalebox{0.9}{${\dv\vd}$}$_{\psi,\sigma}$}\, r_E$.
    \end{proof}
\end{theor}

\smallskip

\begin{cor}\label{corol:diskew_sol_order}
    Let $(D,\vd,\dv,\circ)$ be a di-skew brace, let $r = r_G\, \text{\scalebox{0.9}{${\dv\vd}$}$_{\psi,\sigma}$}\, r_E$ be its associated solution and set
    \begin{align*}
        m_\psi := \textnormal{inf}\left\{h \in \mathbb{N}\mid \forall\, a,b \in D,\, \forall\, v \in E \quad (\psi_b\psi_a)^h\psi_b^{-h}(v) = v\right\} \,.
    \end{align*}
    Then, $r$ has finite order if and only if $G/Z(G)$ is periodic and $m_\psi$ is finite.
    In particular, when this happens $r$ has even order and $o(r) = 2\,\textnormal{lcm}\left(\exp(G/Z(G)),\,m_\psi\right)$.
    \begin{proof}
        The assertion follows after a straightforward application of \cref{corol:hemi_product_sol_order}.
        Indeed, under its notation, it is sufficient to note that $o(r) = 2\,\lcm(m_{\tr_{r_G}}, m_\psi)$ and recall that $m_{\tr_{r_G}} = \exp(G/Z(G))$, as it is clear from \cref{corol:conj_sol_order} and its proof.
    \end{proof}
\end{cor}

The following example shows that the class of di-skew brace solutions is properly larger -- modulo Drinfel'd isomorphisms -- than the class of solutions obtained through skew braces.

\begin{ex}
    Let $D = \Sym_3$ be the symmetric group on three elements and consider the idempotent endomorphism $f$ of $D$ defined by setting ${f(12) = (12)}$ and $f(123) = \id_D$.
    Then, $f$ is an averaging operator on $D$ that determines a $g$-digroup structure $(D,\vd,\dv)$ by setting $a \vd b = f(a) b$ and $a \dv b = a f(b)$, for all $a,b \in D$.
    The latter satisfies $E(D,\vd,\dv) = \ker(f) = \{\id, (123), (132)\}$.
    
    Now, fix a bar-unit $\xi \in E(D)$, the easiest choice being $\xi = \id_D$, and note that if $a \in D$ then $g_a = a \vd \xi = f(a)$.
    Therefore, we deduce that $D$ decomposes as $D = G \,\text{\scalebox{0.9}{$\dv\vd$}}\, E(D)$, where $G = \langle (12) \rangle$ is the groupal component relative to $\xi$.
    Let $(D,\tr) = \Conj(D,\vd,\dv)$ and consider its associated left derived solution $r$.
    A straightforward application of \cref{corol:hemi_product_sol_order} ensures that $o(r) = 4$.
    On the other hand, let $(S,+,\circ)$ be a skew brace of cardinality $6$ and let $s$ be its associated solution.
    Then, $s$ is either involutive or of order $12$, as a consequence of \cref{corol:conj_sol_order}, thus proving that $r$ and $s$ cannot be Drinfel'd isomorphic.
\end{ex}

As a conclusive observation, it is worth underlining the fact that many other aspects on solutions connected to di-skew braces remain open to further detailed analysis.
This will be indeed the subject of future investigations with which we foresee to deepen the relation between our structures and bijective non-degenerate solutions.

\smallskip

\section*{Acknowledgments}
The authors are all members of GNSAGA (INdAM) and the non-profit association ADV-AGTA. This work was supported by the Dipartimento di Matematica e Fisica ``Ennio De Giorgi" - Università del Salento.
A.~Albano was supported by a scholarship financed by the Ministerial Decree no.~118/2023, based on the NRRP - funded by the European Union - NextGenerationEU - Mission 4.

We thank the referee for carefully reading our manuscript and the remarks and suggestions that helped us improve the paper. 

\smallskip

\bibliographystyle{elsart-num-sort}  
\bibliography{bibliography}

\end{document}